\documentclass{article}

\usepackage{amssymb}
\usepackage{amsthm}
\usepackage{amsmath}
\usepackage{mathtools}
\usepackage[all]{xy}

\theoremstyle{definition}
\newtheorem{defn}{Definition}[section]
\theoremstyle{plain}
\newtheorem{lem}[defn]{Lemma}
\newtheorem{prop}[defn]{Proposition}
\newtheorem{thm}[defn]{Theorem}
\newtheorem{cor}[defn]{Corollary}
\theoremstyle{remark}
\newtheorem{rmk}[defn]{Remark}

\newcommand{\id}{{\rm id}}
\newcommand{\pr}{\mathrm{pr}}
\newcommand{\ev}{\mathrm{ev}}
\newcommand{\comp}{\mathrm{comp}}
\newcommand{\incl}{\mathrm{incl}}
\newcommand{\diag}{\Delta}
\renewcommand{\deg}[1]{\mathopen|#1\mathclose|}
\newcommand{\K}{\mathbb K}
\newcommand{\Q}{\mathbb Q}
\newcommand{\Z}{\mathbb Z}
\newcommand{\cochain}[2][*]{C^{#1}(#2)}
\newcommand{\apl}[2][*]{A_\mathrm{PL}^{#1}(#2)}
\newcommand{\cohom}[2][*]{H^{#1}(#2)}
\newcommand{\homol}[2][*]{H_{#1}(#2)}
\newcommand{\im}{\operatorname{Im}}
\renewcommand{\ker}{\operatorname{Ker}}
\renewcommand{\hom}{\mathrm{Hom}}
\newcommand{\ext}{\mathrm{Ext}}
\newcommand{\tor}{\mathrm{Tor}}
\newcommand{\shriek}{\Delta_!}
\newcommand{\repshriek}{\delta_!}
\newcommand{\qis}{\simeq^{\rm q}}
\newcommand{\tpow}[1]{^{\otimes #1}}
\newcommand{\ideal}[2]{(#1)_{#2}}
\newcommand{\felix}{F\'{e}lix}
\renewcommand{\L}{\mathbb L}
\newcommand{\Lcobar}{\mathcal L}
\newcommand{\Cbar}{C}

\newcommand{\fst}[1]{{#1}_{(1)}}
\newcommand{\snd}[1]{{#1}_{(2)}}
\newcommand{\fstandsnd}[2]{\fst{#1}\snd{#2}}
\newcommand{\ith}[1]{{#1}_{(i)}}

\makeatletter
\newcommand{\@tensorpower}[1]{\ifx#1\relax\else^{\otimes #1}\fi}
\newcommand{\@modelcommand}[3]{%
\ifx#2\relax
 {#1}\@tensorpower{#3}
\else
 #1(#2)\@tensorpower{#3}
\fi
}
\newcommand{\@mpath}{{\mathcal M}_\mathrm{P}}
\newcommand{\@mloop}{{\mathcal M}_\mathrm{L}}
\newcommand{\mpath}[1][\relax]{\@modelcommand{\@mpath}{\relax}{#1}}
\newcommand{\mpathv}[2][\relax]{\@modelcommand{\@mpath}{\wedge #2}{#1}}
\newcommand{\mloop}[1][\relax]{\@modelcommand{\@mloop}{\relax}{#1}}
\newcommand{\mloopv}[2][\relax]{\@modelcommand{\@mloop}{\wedge #2}{#1}}

\newcommand{\fga}[2][\relax]{\wedge{#2}\@tensorpower{#1}}
\newcommand{\fdga}[2][\relax]{(\fga{#2}, d)\@tensorpower{#1}}
\makeatother

\newcommand{\dlp}{\mathrm{Dlp}}
\newcommand{\dlcop}{\mathrm{Dlcop}}

\renewcommand{\labelenumi}{{\rm (\arabic{enumi})}}

\makeatletter
\providecommand*{\twoheadrightarrowfill@}{%
  \arrowfill@\relbar\relbar\twoheadrightarrow
}
\providecommand*{\twoheadleftarrowfill@}{%
  \arrowfill@\twoheadleftarrow\relbar\relbar
}
\providecommand*{\xtwoheadrightarrow}[2][]{%
  \ext@arrow 0579\twoheadrightarrowfill@{#1}{#2}%
}
\providecommand*{\xtwoheadleftarrow}[2][]{%
  \ext@arrow 5097\twoheadleftarrowfill@{#1}{#2}%
}
\makeatother

\usepackage{calc}
\newcounter{hours}\newcounter{minutes}
\makeatletter
\renewcommand*{\thehours}{\two@digits\c@hours}
\renewcommand*{\theminutes}{\two@digits\c@minutes}
\makeatother
\newcommand{\printtime}{%
  \setcounter{hours}{\time/60}%
  \setcounter{minutes}{\time-\value{hours}*60}%
  \thehours:\theminutes}
\newcommand{\now}{\the\year/\the\month/\the\day\ \ \printtime}

\title{Description and triviality of the loop products and coproducts for rational Gorenstein spaces}
\author{Shun Wakatsuki}
\date{}

\begin{document}
\maketitle

\begin{abstract}
 \felix\ and Thomas extended the loop products and coproducts to simply-connected Gorenstein spaces.
 We explicitly describe these operations with rational coefficients in terms of Sullivan models.
 Moreover, by this description, we prove some results on triviality of these operations.
 They include a variant of the result of Tamanoi, and generalizations of that of \felix\ and Thomas and that of Naito.
\end{abstract}

\section{Introduction}
\label{sect_intro}
Chas and Sullivan \cite{chas-sullivan} introduced a new operation, called the loop product,
on the homology $\homol{LM}$ of the free loop space of a connected closed oriented manifold $M$.
Constructing a 2-dimensional topological quantum field theory without counit,
Cohen and Godin \cite{cohen-godin} extended this product to other operations, called string operations,
including a coproduct on $\homol{LM}$, which we call the loop coproduct.
But, in many cases, the coproduct is trivial.

\begin{thm}
 [{\cite[Corollary 3.2]{tamanoi}}]
 \label{thm_TrivTamanoi}
 If $M$ is a connected closed oriented manifold with its Euler characteristic zero, then the loop coproduct is trivial.
\end{thm}

Then, generalizing these, \felix\ and Thomas \cite{felix-thomas09} developed a theory of string operations on Gorenstein spaces in an algebraic way.
In this theory,
rational homotopy theory is efficient if the coefficient is a field $\K$ of characteristic zero.
For instance, they proved the following theorem using rational homotopy theory.

\begin{thm}
 [{\cite[Theorem 14]{felix-thomas09}}]
 \label{thm_TrivFT}
 If $G$ is a compact connected Lie group, then the loop product for $M=BG$ with $\K$-coefficients is trivial.
\end{thm}

Moreover, Naito \cite{naito13} proved the following theorem.

\begin{thm}
 [{\cite[Proposition 1.4(2)]{naito13}}]
 \label{thm_TrivNaito}
 Let $M$ be a simply-connected $\K$-Gorenstein space whose minimal Sullivan model $\fdga{V}$ is pure.
 If $\dim V^\mathrm{odd} > \dim V^\mathrm{even}$, then the loop coproduct is trivial.
\end{thm}

Here, we recall the definition of a pure Sullivan algebra.
Let $\fdga{V}$ be a Sullivan algebra.
We denote by $V^{\rm even}$ (resp. $V^{\rm odd}$) the subspace of even (resp. odd) degree elements in $V$.

\begin{defn}
 [c.f. {\cite[Section 32]{felix-halperin-thomas01}}]
 \label{defn_Pure}
 A Sullivan algebra $\fdga{V}$ with $\dim V < \infty$ is called {\it pure}
 if $d(V^{\rm even})=0$ and $d(V^{\rm odd}) \subset \wedge V^{\rm even}$.
\end{defn}


In this paper, we will prove a theorem on triviality of the loop product and coproduct, which consists of a variant of \protect \MakeUppercase {T}heorem\nobreakspace \ref {thm_TrivTamanoi} and generalizations of \protect \MakeUppercase {T}heorem\nobreakspace \ref {thm_TrivFT} and \protect \MakeUppercase {T}heorem\nobreakspace \ref {thm_TrivNaito}.
To study the dual loop product and coproduct, it is important to describe the shriek map $\shriek$ (see Section\nobreakspace \ref {sect_StringTopology} for the definition).
In Section\nobreakspace \ref {sect_ConstructionOfShriek}, we will give an explicit construction of the map by means of Sullivan models, which generalizes Naito's construction for pure Sullivan algebras \cite{naito13}.
For this construction, we introduce the notion of a semi-pure Sullivan algebra as a generalization of a pure Sullivan algebra.
Let $I_V$ be the ideal in $\wedge V$ generated by $V^{\rm even}$.

\begin{defn}
 \label{defn_Semipure}
 A Sullivan algebra $\fdga{V}$ with $\dim V < \infty$ is called {\it semi-pure}
 if it satisfies $d(V^{\rm even}) \subset I_V$, or equivalently $d(I_V) \subset I_V$.
\end{defn}

The semi-purity is not essential by the following theorem.

\begin{thm}
 \label{thm_Semi-pureModel}
 Let $\fdga{V}$ be a Sullivan algebra satisfying $\dim V < \infty$ and $V^1 = 0$.
 Then there is a semi-pure Sullivan algebra $\fdga{W}$ satisfying $\dim W < \infty$, $W^1=0$ and $\fdga{V}\simeq\fdga{W}$.
\end{thm}

Note that the above theorem remains true if we assume $V$ is of finite type instead of assuming $\dim V < \infty$ in the theorem and \protect \MakeUppercase {D}efinition\nobreakspace \ref {defn_Pure}.
We defer the proof of the theorem to Appendix\nobreakspace \ref {sect_pure},
since we use this theorem only in the proof of \protect \MakeUppercase {T}heorem\nobreakspace \ref {thm_main} (\ref{item_TrivNaito}).

Now, let us give our theorem on triviality of the loop product and coproduct. 
Since \felix\ and Thomas used an algebraic method,
we can define the dual loop product and coproduct for a Sullivan algebra $\fdga{V}$, which are the duals of the loop product and coproduct for $M$ if $\fdga{V}$ is a Sullivan model of a Gorenstein space $M$.
We denote these by $\dlp_{\fga{V}}$ and by $\dlcop_{\fga{V}}$, respectively.

\begin{thm}
 \label{thm_main}
 Let $\fdga{V}$ be a Sullivan algebra with $\dim V < \infty$ and $V^1=0$. 
 \begin{enumerate}
  \item \label{item_TrivTamanoi}
	Assume that $\fdga{V}$ is semi-pure and that there is a direct sum decomposition $V = \K x \oplus W$ such that $\fdga{V} = (\fga{x} \otimes \fga{W}, d)$ is a relative Sullivan algebra over a Sullivan algebra $(\fga{x}, 0)$ of one generator with $\deg{x}$ odd.
	Then $\dlcop_{\fga{V}}$ is trivial.
  \item \label{item_TrivFT}
	Assume that $\fdga{V}$ is semi-pure and that there is a direct sum decomposition $V = W \oplus \K x$ such that $\fdga{V} = (\fga{W} \otimes \fga{x}, d)$ is a relative Sullivan algebra over a Sullivan algebra $\fdga{W}$ with $\deg{x}$ even.
	Then $\dlp_{\fga{V}}$ is trivial.
  \item \label{item_TrivNaito}
	If $\dim V^{\mathrm{even}} < \dim V^\mathrm{odd}$,
	then $\dlcop_{\fga{V}}$ is trivial.
 \end{enumerate}
\end{thm}

In the above theorem, (\ref{item_TrivTamanoi}) is related to \protect \MakeUppercase {T}heorem\nobreakspace \ref {thm_TrivTamanoi}, and (\ref{item_TrivFT}) and (\ref{item_TrivNaito}) are generalizations of \protect \MakeUppercase {T}heorem\nobreakspace \ref {thm_TrivFT} and \protect \MakeUppercase {T}heorem\nobreakspace \ref {thm_TrivNaito}, respectively.
Roughly speaking, the assumption of (\ref{item_TrivTamanoi}) is analogous to saying that $\fdga{V}$ is a Sullivan model of the total space of a fibration $F\rightarrow M \rightarrow S^{2k+1}$ whose base space is (rationally homotopy equivalent to) an odd dimensional sphere, since a relative Sullivan algebra is a model of a fibration.
Similarly, the assumption of (\ref{item_TrivFT}) is roughly stated with a fibration $K(\Z, 2k) \rightarrow M \rightarrow B$ whose fiber is (rationally homotopy equivalent to) an Eilenberg-MacLane space of type $(\Z, 2k)$.

Note that we do not assume the minimality of a Sullivan algebra in the construction of the shriek map and in \protect \MakeUppercase {T}heorem\nobreakspace \ref {thm_main}.
Hence, for a Sullivan algebra $\fdga{V}$ with $\dim V < \infty$ and $V^1 = 0$,
we can apply them if we replace $\fdga{V}$ with a semi-pure Sullivan algebra $\fdga{W}$ by \protect \MakeUppercase {T}heorem\nobreakspace \ref {thm_Semi-pureModel}.

\tableofcontents

\section{Rational homotopy theory}
\label{sect_RationalHomotopyTheory}
We recall basic definitions and theorems in rational homotopy theory.
See \cite{felix-halperin-thomas01} for more details on this section.


In this article, all modules and algebras are defined on a field $\K$ of characteristic zero.
A \textit{differential graded algebra} is a pair $(A,d)$ of graded algebra $A$ and a derivation $d$ satisfying $d^2=0$.
For simplicity of notation, we write $A$ instead of $(A,d)$ in diagrams, $\hom$, $\tor$, or $\ext$.
We abbreviate a differential graded algebra to a dga, and a commutative dga to a cdga.

For an element $x$ of a graded algebra or a graded module, we denote by $\deg{x}$ the degree of $x$.
For a graded algebra $A$ and its elements $a_1, \ldots, a_n$, we denote by $\ideal{a_1,\ldots, a_n}{A}$ the ideal generated by $a_1,\ldots,a_n$ in $A$.
Let $V$ be a graded module.
We denote by $\fga{V}$ the free graded commutative algebra on $V$.
If $W$ is also a graded module, we identify $\fga{(V\oplus W)} = \fga{V}\otimes\fga{W}$ canonically.
We denote by $\overline{V}$ the graded module defined by $\overline{V}^n = V^{n+1}$,
and we denote by $\bar{v}$ the element of $\overline{V}$ corresponding to an element $v \in V$.

A dga homomorphism $f\colon (A,d) \rightarrow (B,d)$ is called a \textit{quasi-isomorphism} if $f$ induces isomorphism on the cohomology.
Then we denote $f\colon (A,d) \xrightarrow{\qis} (B,d)$.
Two dga's $(A,d)$ and $(B,d)$ are quasi-isomorphic if these are connected by a sequence of quasi-isomorphisms
\[
 (A,d) \xrightarrow{\qis} (A_0,d) \xleftarrow{\qis} (A_1,d) \xrightarrow{\qis} \cdots \xrightarrow{\qis} (A_n, d) \xleftarrow{\qis} (B, d).
\]

For a topological space $X$, we denote by $\cochain{X}$ the normalized singular cochain algebra of $X$ with coefficients in $\K$, and by $\cohom{X}$ its cohomology.

\begin{thm}
 [{\cite[Corollary 10.10]{felix-halperin-thomas01}}]
 \label{thm_Apl}
 For any topological space $X$, there is a cdga $\apl{X}$ which is naturally quasi-isomorphic to $\cochain{X}$.
 More precisely, there is another dga $D^*(X)$ with natural dga quasi-isomorphisms
 \[
 \cochain{X} \xrightarrow{\qis} D^*(X) \xleftarrow{\qis} \apl{X}.
 \]
\end{thm}

Recall the definitions of a Sullivan algebra and a relative Sullivan algebra.
A \textit{Sullivan algebra} is a cdga of the form $\fdga{V}$, where
\begin{itemize}
 \item $V = \{V^p\}_{p\geq1}$ is a graded $\K$-module concentrated in positive degrees, and
 \item there is a filtration $\{V(k)\}_{k\geq-1}$ such that
	\[
	 0=V(-1) \subset V(0) \subset V(1) \subset \cdots \subset V
	\]
	and $d(V(k))\subset \fga{V(k-1)}$.
\end{itemize}
A Sullivan algebra $\fdga{V}$ is called \textit{minimal} if $d(V) \subset \wedge^{\geq2}V$.
A \textit{Sullivan model} for a cdga $(A,d)$ is a quasi-isomorphism
\[
m\colon \fdga{V} \xrightarrow{\qis} (A,d)
\]
from a Sullivan algebra.
It is \textit{minimal} if $\fdga{V}$ is a minimal Sullivan algebra.
A \textit{Sullivan model} for a path-connected topological space $X$ is a Sullivan model
\[
 m\colon \fdga{V} \rightarrow \apl{X}
\]
for $\apl{X}$.

A \textit{relative Sullivan algebra with base $(B,d)$} is a cdga of the form $(B\otimes\fga{V}, d)$, where
\begin{itemize}
 \item $(B,d) = (B\otimes\K,d)$ is a sub cdga with $\cohom{B,d} = \K$,
 \item $\K\otimes V = V = \{V^p\}_{p\geq1}$ is a graded $\K$-module concentrated in positive degrees, and
 \item there is a filtration $\{V(k)\}_{k\geq-1}$ such that
	\[
	 0=V(-1) \subset V(0) \subset V(1) \subset \cdots \subset V
	\]
	and $d(V(k))\subset B\otimes\fga{V(k-1)}$.
\end{itemize}
We always write the base algebra $B$ to the left of the symbol $\otimes$.
A relative Sullivan algebra $(B\otimes\fga{V}, d)$ is called \textit{minimal} if $d(V) \subset B^{+}\otimes\fga{V} + B\otimes\wedge^{\geq2}V$.
Let $\varphi\colon (B,d) \rightarrow (C,d)$ be a cdga homomorphism with $\cohom{B,d}=\cohom{C,d}=\K$.
A \textit{(relative) Sullivan model} for the homomorphism $\varphi$ is a quasi-isomorphism
\[
 m\colon (B\otimes\fga{V},d) \xrightarrow{\qis} (C,d)
\]
from a relative Sullivan algebra with base $(B,d)$ such that the restriction $m|_B$ is equal to $\varphi$.
It is \textit{minimal} if $(B\otimes\fga{V},d)$ is a minimal relative Sullivan algebra.
A \textit{(relative) Sullivan model} for a map $f\colon X\rightarrow Y$ of path-connected topological spaces is a Sullivan model for $\apl{f}$.

Then, the existence of a (relative) Sullivan model is given by the following proposition.

\begin{prop}
 [{\cite[Proposition 12.1, Proposition 14.3]{felix-halperin-thomas01}}]
 \label{prop_SullModel}
 \quad
 \begin{enumerate}
  \item	Any cdga $(A,d)$ with $\cohom[0]{A,d} = 0$ and any path-connected topological space $X$ has a Sullivan model.
  \item	Any cdga homomorphism $\varphi\colon (B,d)\rightarrow(C,d)$ with $\cohom[0]{B,d}=\cohom[0]{C,d}=\K$ and $\cohom[1]{\varphi}$ injective 
	has a relative Sullivan model.
 \end{enumerate}
\end{prop}

Moreover, \textit{minimal} (relative) Sullivan models are constructed as sub cdga's of (relative) Sullivan models.

\begin{thm}
 [{\cite[Theorem 14.9]{felix-halperin-thomas01}}]
 \label{thm_MinimalModel}
 Let $(B\otimes\fga{V}, d)$ be a relative Sullivan algebra.
 Then there is a graded subspace $W$ of $V$ and a differential $d'$ on $B\otimes\fga{W}$
 such that $(B\otimes\fga{W}, d')$ is a minimal relative Sullivan algebra and the map
 \[
  (B\otimes\fga{W}, d')\rightarrow (B\otimes\fga{V}, d)
 \]
 induced by the inclusions is a cdga isomorphism.
\end{thm}

\begin{cor}
 \label{cor_MinimalModel}
 \ 
 \begin{enumerate}
  \item Any cdga $(A,d)$ with $\cohom[0]{A,d} = 0$ and any path-connected topological space $X$ has a minimal Sullivan model.
  \item Any cdga homomorphism $\varphi\colon (B,d)\rightarrow(C,d)$ with $\cohom{B,d}=\cohom{C,d}=\K$ and $\cohom[1]{\varphi}$ injective 
	has a minimal relative Sullivan model.
 \end{enumerate}
\end{cor}

Recall a basic lemma on quasi-isomorphisms of relative Sullivan algebras, which will be needed in the following sections.

\begin{lem}
 [{\cite[Lemma 14.2]{felix-halperin-thomas01}}]
 \label{lem_QisOfRelSullAlg}
 Let $(B,d), (C,d)$ be cdga's with $\cohom{B} = \cohom{C} = \K$, $(B\otimes\wedge V, d)$ a relative Sullivan algebra over $(B,d)$, and $m\colon (B,d) \xrightarrow{\qis} (C,d)$ a quasi-isomorphism of cdga's.
 Define a relative Sullivan algebra over $(C,d)$ by $(C\otimes\wedge V, d) = (C,d) \otimes_{B} (B\otimes\wedge V, d)$, where $(C,d)$ is a $(B,d)$-module via $m$.
 Then the cdga homomorphism $n\colon (B\otimes\wedge V, d) \rightarrow (C\otimes\wedge V, d)$ defined by $m \otimes \id$
 is a quasi-isomorphism.
\end{lem}

Recall the following important theorems in rational homotopy theory. 
By these theorems, the results in this article can be understood in a topological context.

\begin{thm}
 [{\cite[Theorem 15.11]{felix-halperin-thomas01}}]
 \label{thm_IsomWithHomotopyGroup}
 Let $X$ be a simply connected topological space such that $\homol{X;\K}$ is of finite type,
 and $\fdga{V}$ the minimal Sullivan model of $X$.
 Then, the graded module $V$ is isomorphic to the homotopy group $\pi_*(X)\otimes\K$.
\end{thm}

\begin{thm}
 [{\cite[Theorem 17.10]{felix-halperin-thomas01}}]
 \label{thm_Realization}
 Assume that the coefficient field $\K$ is the field $\Q$ of rational numbers.
 Let $\fdga{V}$ be a Sullivan algebra such that $\cohom[1]{\fga{V},d}=0$ and $\cohom{\fga{V},d}$ is of finite type.
 Then there is a simply connected topological space $\mathopen|\fga{V},d\mathclose|$ such that $\fdga{V}$ is a Sullivan model of $\mathopen|\fga{V},d\mathclose|$.
\end{thm}

Now, recall the definitions of $\tor$ and $\ext$.
Let $(R,d)$ be a dga, $(M,d)$ and $(N,d)$ left $(R,d)$-modules, and $(L,d)$ right $(R,d)$-module.
Then we define
\[
\tor_R^n(L, M) = \cohom[n]{L\otimes_RP}
\]
and 
\[
\ext_R^n(M, N) = \cohom[n]{\hom_R(P, N)},
\]
where $(P,d)$ is a semifree resolution of $(M,d)$, i.e. $(P,d)$ is a semifree $(R,d)$-module with a quasi-isomorphism $(P,d)\xrightarrow{\qis}(M,d)$ of $(R,d)$-modules.
For the definitions of an $(R,d)$-module and a semifree $(R,d)$-module, see Section 3 and 6 of \cite{felix-halperin-thomas01}, respectively.
Note that a relative Sullivan algebra $(B\otimes\fga{V}, d)$ is a semifree $(B,d)$-module.
Hence a relative Sullivan model can be used as a semifree resolution.

\section{String topology}
\label{sect_StringTopology}
Recall the definition of a Gorenstein space.
\begin{defn}
 [{\cite{felix-halperin-thomas88}}]
 \label{defn_Gorenstein}
 Let $m\in\Z$ be an integer.
 \begin{enumerate}
  \item An augmented dga $(A,d)$ is called a \textit{($\K$-)Gorenstein algebra of dimension} $m$ if
	\[
	\dim \ext_A^l(\K, A) = \begin{cases}
				1 \mbox{ (if $l = m$)} \\
				0 \mbox{ (otherwise),}
			       \end{cases}
	\]
	where the field $\K$ and the dga $(A,d)$ are $(A,d)$-modules via the augmentation map and the identity map, respectively.
  \item A path-connected topological space $M$ is called a \textit{($\K$-)Gorenstein space of dimension} $m$
	if $\cochain{M}$ is a Gorenstein algebra of dimension $m$.
 \end{enumerate}
\end{defn}

An important example of a Gorenstein space is given by the following theorem.

\begin{thm}
 [{\cite[Proposition 3.4]{felix-halperin-thomas88}}]
 \label{thm_FinDimImplyGorensteinSpace}
 A simply-connected topological space $X$ is a $\Q$-Gorenstein space if $\pi_*(X)\otimes\Q$ is finite dimensional.
\end{thm}

In this paper, we study the dual loop product and coproduct for a Sullivan algebra $\fdga{V}$ with $\dim V < \infty$,
which is always a Gorenstein algebra by the following theorem.

\begin{thm}
 \label{thm_FinDimImplyGorensteinAlg}
 A Sullivan algebra $\fdga{V}$ (over $\K$) is a Gorenstein algebra if $V$ is finite dimensional.
\end{thm}
\begin{proof}
 This can be proved by a similar method to $\protect \MakeUppercase {T}heorem\nobreakspace \ref {thm_FinDimImplyGorensteinSpace}$.
\end{proof}

Let $M$ be a simply connected $\K$-Gorenstein space of dimension $m$ whose cohomology $\cohom{M}$ is of finite type.
As a preparation to define the loop product and coproduct, \felix-Thomas proved the following theorem.

\begin{thm}
 [{\cite[Theorem 12]{felix-thomas09}}]
 \label{thm_ExtDiagonal}
 Let $n$ be a positive integer.
 The diagonal map $\Delta\colon M \rightarrow M^n$ makes $\cochain{M}$ into a $\cochain{M^n}$-module.
 Then we have an isomorphism
 \[
 \ext_{\cochain{M^n}}^*(\cochain{M}, \cochain{M^n}) \cong \cohom[*-(n-1)m]{M}.
 \]
\end{thm}

By \protect \MakeUppercase {T}heorem\nobreakspace \ref {thm_ExtDiagonal}, we have $\ext_{\cochain{M^2}}^m(\cochain{M}, \cochain{M^2})\cong\cohom[0]{M}\cong\K$, hence the generator
\[
 \shriek \in \ext_{\cochain{M^2}}^m(\cochain{M}, \cochain{M^2})
\]
is well-defined up to the multiplication by a non-zero scalar.
We call this element the \textit{shriek map} for $\Delta$.
For the loop product, consider the diagram
\[
 \xymatrix{
 LM & LM \times_M LM \ar[l]_-{\comp} \ar[r]^-{\incl} \ar[d]^{\ev_0}& LM \times LM \ar[d]^{\ev_0\times\ev_0} \\
 & M \ar[r]^-{\diag} & M \times M,
 }
\]
where the maps comp, incl, and diag are the composition map, the inclusion map, and the diagonal map, respectively,
and the right square is a pullback square.
Then define the dual loop product by the composition
\[
 \dlp\colon \cohom{LM}
 \xrightarrow{\comp^*} \cohom{LM\times_MLM}
 \xrightarrow{\incl_!} \cohom{LM\times LM}
 \xleftarrow[\cong]{\times} \cohom{LM}\tpow2.
\]
Here, the map $\incl_!$ is defined by the composition
\[
 \begin{array}{l}
  \cohom{LM\times_M LM}
   \xleftarrow[\cong]{\mathrm{EM}} \tor_{\cochain{M^2}}(\cochain{M},\cochain{LM\times LM}) \\
  \xrightarrow{\tor_\id(\shriek, \id)}\tor_{\cochain{M^2}}(\cochain{M^2},\cochain{LM\times LM})
   \xrightarrow[\cong]{} \cohom{LM\times LM},
 \end{array}
\]
where the map $\mathrm{EM}$ is the Eilenberg-Moore map, which is an isomorphism (see \cite[Section 7]{felix-halperin-thomas01} for details).

For the loop coproduct, consider the diagram
\[
 \xymatrix{
 LM\times LM & LM \times_M LM \ar[l]_-{\incl} \ar[r]^-{\comp} \ar[d]^{\ev_0} & LM \ar[d]^{l}\\
 & M \ar[r]^-{\diag} & M^2,
 }
\]
where the map $l$ is defined by $l(\gamma) = (\gamma(0), \gamma(\frac{1}{2}))$ for $\gamma\in LM$, 
and the right square is a pullback square.
Then define the dual loop coproduct by the composition
\[
 \dlcop\colon \cohom{LM}\tpow2
 \xrightarrow[\cong]{\times} \cohom{LM\times LM}
 \xrightarrow{\incl^*} \cohom{LM\times_M LM}
 \xrightarrow{\comp_!} \cohom{LM},
\]
where the map $\comp_!$ is defined by the composition
\[
\begin{array}{l}
 \cohom{LM\times_M LM}
 \xleftarrow[\cong]{\mathrm{EM}} \tor_{\cochain{M^2}}(\cochain{M}, \cochain{LM})\\
 \xrightarrow{\tor_\id(\shriek, \id)} \tor_{\cochain{M^2}}(\cochain{M^2}, \cochain{LM})
 \xrightarrow[\cong]{} \cohom{LM}.
\end{array}
\]

Let $\fdga{V}$ be a Sullivan model of the Gorenstein space $M$, 
and $\mu_{\wedge V}\colon \fdga{V}^{\otimes 2} \rightarrow \fdga{V}$ the multiplication map of $\fdga{V}$.
Take a relative Sullivan model $\bar{\varepsilon}\colon (\wedge V\tpow 2 \otimes \wedge U, d) \xrightarrow{\qis} (\wedge V, d)$ of $\mu_{\wedge V}$.
We denote $(\mpath, d) = (\wedge V\tpow 2 \otimes \wedge U, d)$ and $(\mloop, d) = (\wedge V \otimes \wedge U, d) = (\wedge V, d) \otimes_{\wedge V\tpow{2}} (\wedge V \tpow{2} \otimes \wedge U, d)$.
For $v\in V$ and $u\in U$, we denote by $\fst{v}$, $\snd{v}$, and $u$ the elements $v\otimes 1\otimes 1$, $1\otimes v\otimes 1$, and $1\otimes 1\otimes u$ of $\mpath = \fga[2]{V}\otimes\fga{U}$, respectively.
Similarly, for $v\in V$ and $u\in U$, we denote by $v$ and $u$ the elements $v\otimes 1$ and $1\otimes u$ of $\mloop = \fga{V}\otimes\fga{U}$,
and, for $v \in V$, we denote by $\fst{v}$ and $\snd{v}$ the elements $v\otimes 1$ and $1\otimes v$ of $\fga[2]{V}$.
We use these notations throughout this article.

In \cite{naito13}, the dual loop product and coproduct are described in terms of Sullivan models using the torsion functor description of \cite{kuribayashi-menichi-naito}.
The dual loop product is induced by the following composite
\[
\begin{array}{l}
 \mloop
  \xrightarrow{\cong} \wedge V \otimes_{\wedge V\tpow{2}} \mpath
  \xleftarrow[\qis]{\bar{\varepsilon}\otimes \id} \mpath \otimes_{\wedge V\tpow{2}} \mpath
  \xrightarrow{(\mu\otimes\id)\otimes_\mu(\mu\otimes\id)} \mloop \otimes_{\wedge V} \mloop\\
 \xrightarrow{\cong} \wedge V \otimes_{\wedge V\tpow{2}}\mloop[2]
  \xleftarrow[\qis]{\bar{\varepsilon}\otimes\id} \mpath\otimes_{\wedge V\tpow{2}}\mloop[2]
  \xrightarrow{\shriek\otimes\id} \wedge V\tpow{2}\otimes_{\wedge V\tpow{2}}\mloop[2]
  \xrightarrow{\cong} \mloop[2],
\end{array}
\]
and the dual loop coproduct is induced by the following composite
\[
 \begin{array}{l}
  \mloop[2]
   \xrightarrow{\cong} \wedge V\tpow{2} \otimes_{\wedge V\tpow{4}} \mpath[2]
   \xrightarrow{\mu\otimes_{\mu'}\bar{\zeta}} \wedge V \otimes_{\wedge V\tpow{2}} (\mpath\otimes_{\wedge V\tpow{2}}\mpath)\\
  \xleftarrow[\qis]{\bar{\varepsilon}\otimes\id} \mpath\otimes_{\wedge V\tpow{2}}(\mpath\otimes_{\wedge V\tpow{2}}\mpath)
   \xrightarrow{\shriek\otimes\id} \wedge V\tpow{2}\otimes_{\wedge V\tpow{2}}(\mpath\otimes_{\wedge V\tpow{2}}\mpath)\\
  \xrightarrow{\cong} \mpath\otimes_{\wedge V\tpow{2}}\mpath
   \xrightarrow[\qis]{\bar{\varepsilon}\otimes\id} \wedge V \otimes_{\wedge V\tpow{2}} \mpath
   \xrightarrow{\cong} \mloop.
 \end{array}
\]
Here, see \cite{naito13} for the definitions of the maps in the above diagrams.
By this composition, we can \textit{define} the dual loop product and coproduct for a Sullivan algebra.
We denote these by $\dlp_{\wedge V}\colon \cohom{\mloop} \rightarrow \cohom{\mloop}\tpow{2}$ and $\dlcop_{\wedge V}\colon \cohom{\mloop}\tpow{2} \rightarrow \cohom{\mloop}$.

To study $\dlp_{\fga{V}}$ and $\dlcop_{\fga{V}}$, we will construct $\shriek$ explicitly in Section\nobreakspace \ref {sect_ConstructionOfShriek}.

\section{Construction of a relative Sullivan model for the multiplication map}
\label{sect_ModelOfMultiplication}
In this section, let $\fdga{V}$ be a Sullivan algebra satisfying $\dim V < \infty$ and $V^1 = 0$.
We will construct a relative Sullivan model for the multiplication map $\mu_{\wedge V}$ of $\fdga{V}$ inductively on $\dim V$.
This construction is similar to that of F\'{e}lix, Opera and Tanr\'{e} in \cite[Example 2.48]{felix-opera-tanre}, but is a generalization in the point that we do not assume the minimality of a Sullivan algebra.

\begin{rmk}
 Note that, although we assume $\dim V < \infty$ in this section,
 we can construct the relative Sullivan model for any Sullivan algebra $\fdga{V}$ with $V^1=0$ by taking the colimit.
\end{rmk}

To construct the relative Sullivan model, we need some lemmas.

\begin{lem}
 \label{lem_CoboundByKernel}
 Let $f\colon (C,d) \xtwoheadrightarrow{\qis} (D,d)$ be a surjective quasi-isomorphism of complexes.
 If $c \in C$ satisfies $fc = 0$ and $dc = 0$,
 we have $dc' = c$ for some $c'\in C$ with $fc'=0$.
\end{lem}
\begin{proof}
 Since $f$ is surjective,
 \[
  0 \rightarrow (\ker f, d) \rightarrow (C,d) \xrightarrow{f} (D,d) \rightarrow 0
 \]
 is a short exact sequence of complexes.
 Then $\cohom{\ker f, d} = 0$ since $f$ is a quasi-isomorphism.
 Hence the cocycle $c \in \ker f$ is a coboundary in $\ker f$.
\end{proof}

\begin{lem}
 \label{lem_QisElementary}
 Let $(B\otimes\wedge x, d)$ be a relative Sullivan algebra with base $(B,d)$ satisfying $\deg{x}\geq 2$.
 Define a relative Sullivan algebra $(B\otimes\wedge y\tpow{2}\otimes\wedge\bar{y}, d)$ with base $(B,d)$ by
 $d(\snd{y}) = d(\fst{y}) = dx$ and $d\bar{y} = \snd{y} - \fst{y}$, where $dx\in B$ is considered as an element of $B\otimes\wedge y\tpow{2}\otimes\wedge\bar{y}$ by the canonical inclusion,
 and a dga homomorphism
 \[
 \beta\colon (B\otimes\wedge y\tpow{2}\otimes\wedge\bar{y}, d) \rightarrow (B\otimes\wedge x, d)
 \]
 by $\beta(b) = b$ for $b\in B$, $\beta(\snd{y}) = \beta(\fst{y}) = x$, and $\beta(\bar{y})=0$.
 Then $\beta$ is a quasi-isomorphism.
\end{lem}
\begin{proof}
 Define 
 \[
  \gamma\colon (B \otimes \wedge x, d) \rightarrow (B \otimes \wedge y\tpow{2} \otimes\wedge\bar{y}, d)
 \]
 by $\gamma(b) = b$ for $b\in B$ and $\gamma(x) = \snd{y}$,
 and
 \[
  h\colon (B \otimes \wedge y\tpow{2} \otimes\wedge\bar{y}, d) \rightarrow (B \otimes \wedge y\tpow{2} \otimes\wedge\bar{y}, d) \otimes \wedge(t, dt)
 \]by $h(b)=b$ for $b \in B$, $h(\snd{y}) = \snd{y}$, $h(\fst{y}) = (\fst{y})(1-t) + (\snd{y})t - (-1)^{\deg{y}}\bar{y}dt$, and $h(\bar{y}) = \bar{y}(1-t)$.
 Then they satisfy $\beta\gamma=\id$ and $\gamma\beta \simeq \id$ rel $B$ via the homotopy $h$, and hence $\beta$ is a homotopy equivalence.
 In particular it is a quasi-isomorphism.
\end{proof}

We construct a relative Sullivan model of the multiplication map inductively using the following proposition.

\begin{prop}
 \label{prop_ConstructionOfRelSullModel}
 Let $(\wedge V \otimes \wedge x, d)$ be a relative Sullivan algebra with base $\fdga{V}$ such that $\deg{x} \geq 2$.
 Take a relative Sullivan model
 \[
  m\colon (\wedge V^{\otimes 2}\otimes \wedge U, d) \xrightarrow{\qis} (\wedge V, d)
 \]
 for $\mu_{\wedge V}$.
 Then the followings hold.
 \begin{enumerate}
  \item There is an element $x' \in \wedge V ^{\otimes 2} \otimes \wedge U$ such that $dx' = \snd{dx} - \fst{dx}$ and $mx' = 0$.
  \item Let
	$((\wedge V \otimes \wedge x)\tpow{2} \otimes \wedge U \otimes \wedge \bar{x}, d)$
	be a relative Sullivan algebra with base
	$((\wedge V \otimes \wedge x)\tpow{2} \otimes \wedge U, d)$
	defined by $\deg{\bar{x}} = \deg{x} - 1$ and
	$d\bar{x} = \snd{x} - \fst{x} - x'$.
	Then, considering it as a relative Sullivan algebra
	$((\wedge V \otimes \wedge x)\tpow{2} \otimes (\wedge U \otimes \wedge \bar{x}), d)$
	with base $(\wedge V \otimes \wedge x, d)\tpow{2}$
	, define the map
	\[
	 n\colon ((\wedge V \otimes \wedge x)\tpow{2} \otimes (\wedge U \otimes \wedge \bar{x}), d) \xrightarrow{\qis} (\wedge V \otimes \wedge x, d)
	\]
	by $n(t) = m(t)$ for $t \in V \oplus U$,
	           $n(\fst{x}) = n(\snd{x}) = x$,
	           and $n(\bar{x}) = 0$.
	Then $n$
	is a relative Sullivan model for
	$\mu_{\wedge V \otimes \wedge x}$.
 \end{enumerate}
\end{prop}
\begin{proof}
 (1)
 Since $dx \in \wedge V$, we have $m(\snd{dx} - \fst{dx}) = dx - dx = 0$.
 Hence, by \protect \MakeUppercase {L}emma\nobreakspace \ref {lem_CoboundByKernel}, some $x'$ satisfies the condition.

 (2)
 Because of the property $dx' = \snd{dx} - \fst{dx}$,
 we can define a cdga $((\wedge V \otimes \wedge x)\tpow{2} \otimes (\wedge U \otimes \wedge \bar{x}), d)$, and it is a Sullivan algebra.
 Since $m(x') = 0$, the map $n$ is a cdga homomorphism.
 It follows immediately from the definition
 that the restriction of $n$ to $(\wedge V \otimes \wedge x)\tpow{2}$
 is equal to $\mu_{\wedge V \otimes \wedge x}$

 Hence it is enough to show that $n$ is a quasi-isomorphism.
 Define a relative Sullivan algebra with base $\fdga{V}$ by
 \[
  (\wedge V \otimes {\wedge x}\tpow{2} \otimes \wedge\bar{x}, d) = \fdga{V} \otimes_{{\wedge V}\tpow{2}\otimes\wedge U} ((\wedge V\tpow{2} \otimes \wedge U) \otimes (\wedge x\tpow{2} \otimes \wedge \bar{x}), d),
 \]
 where $\fdga{V}$ is a $(\wedge V\tpow{2} \otimes \wedge U, d)$-module via $m$,
 and $((\wedge V\tpow{2} \otimes \wedge U) \otimes (\wedge x\tpow{2} \otimes \wedge \bar{x}), d)$
 is canonically identified with
 $((\wedge V \otimes \wedge x)\tpow{2} \otimes (\wedge U \otimes \wedge \bar{x}), d)$.
 Then the map
 \[
 \alpha \colon ((\wedge V\tpow{2} \otimes \wedge U) \otimes (\wedge x\tpow{2} \otimes \wedge \bar{x}), d) \rightarrow (\wedge V \otimes {\wedge x}\tpow{2} \otimes \wedge\bar{x}, d)
 \]
 defined by $m\otimes\id$ is a quasi-isomorphism by \protect \MakeUppercase {L}emma\nobreakspace \ref {lem_QisOfRelSullAlg}.
 Define a cdga homomorphism
 \[
  \beta\colon (\wedge V \otimes \wedge x\tpow{2} \otimes \wedge\bar{x}, d) \rightarrow (\wedge V \otimes x, d)
 \]
 by $\beta(v) = v$ for $v\in V$, $\beta(\fst{x}) = \beta(\snd{x}) = x$,
 and $\beta(\bar{x}) = 0$.
 Then $\beta$ is a quasi-isomorphism by \protect \MakeUppercase {L}emma\nobreakspace \ref {lem_QisElementary}.
 Hence $n=\beta\alpha$ is also a quasi-isomorphism.
 This proves the proposition.
 \[
 \xymatrix@C=0pt{
 (\wedge V \otimes \wedge x)\tpow{2} \otimes (\wedge U \otimes \wedge \bar{x}) \ar[d]^\cong \ar[rr]^-n
 && \wedge V \otimes \wedge x \ar[d]^=\\
 (\wedge V\tpow{2} \otimes \wedge U) \otimes (\wedge x\tpow{2} \otimes \wedge \bar{x}) \ar[rr]^-{n} \ar[dr]_\alpha
 && \wedge V \otimes \wedge x \\
 &\wedge V \otimes (\wedge x\tpow{2} \otimes \wedge\bar{x}) \ar[ur]_\beta
 }
 \]
%

\end{proof}

Recall, from Section\nobreakspace \ref {sect_RationalHomotopyTheory}, that $\overline{V}$ is a graded module defined by $\overline{V}^n = V^{n+1}$.

\begin{cor}
 \label{cor_ConstructionOfRelSullModel}
 There is a relative Sullivan model for $\mu_{\wedge V}$ of the form
 \[
  m\colon (\wedge V\tpow 2\otimes \wedge \overline{V}, d) \xrightarrow{\qis} (\wedge V, d).
 \]
\end{cor}
\begin{proof}
 Use \protect \MakeUppercase {P}roposition\nobreakspace \ref {prop_ConstructionOfRelSullModel} inductively on $\dim V$.
\end{proof}

\begin{rmk}
 We can explicitly construct the element $x'$ in \protect \MakeUppercase {P}roposition\nobreakspace \ref {prop_ConstructionOfRelSullModel} if $U = \overline{V}$.
 See Appendix\nobreakspace \ref {sect_appendix} for details.
\end{rmk}

Recall the ideal $I_V = \ideal{x_1, \ldots, x_p}{\wedge V}$ defined in Section\nobreakspace \ref {sect_intro}.
Similarly, define ideals $\bar{I}_V$ and $\bar{I}_{V,i}$ by
\[
 \bar{I}_V = \ideal{\fst{x_1}, \ldots, \fst{x_p}, \snd{x_1}, \ldots, \snd{x_p}, \bar{x}_1 ,\ldots, \bar{x}_p}{\wedge V\tpow 2 \otimes \wedge \overline{V}}
\]
and
\[
 \bar{I}_{V,i} = \ideal{\fst{x_1}, \ldots, \fst{x_p}, \snd{x_1}, \ldots, \snd{x_p}, \bar{x}_1 ,\ldots, \bar{x}_{i-1}, \bar{x}_{i+1}, \ldots\bar{x}_p}{\wedge V\tpow 2 \otimes \wedge \overline{V}}
\]
for $1 \leq i \leq p$.
Note that
statements such as ``$\bar{I}_V$ is a differential ideal'' do not make sense,
since a differential on $\wedge V\tpow 2 \otimes \wedge \overline{V}$ is not yet specified.

Recall, from Section\nobreakspace \ref {sect_intro}, that a semi-pure Sullivan algebra $\fdga{V}$ is a Sullivan algebra with $\dim V < \infty$ and $d(I_V)\subset I_V$.

\begin{defn}
 \label{defn_NiceRelSullModel}
 Let $\fdga{V}$ be a semi-pure Sullivan algebra.
 A relative Sullivan model
 $m\colon (\wedge V\tpow 2\otimes \wedge \overline{V}, d) \xrightarrow{\qis} \fdga{V}$
 for $\mu_{\wedge V}$ is {\it nice} if it satisfies the following conditions:
 \begin{enumerate}
  \renewcommand{\labelenumi}{(\alph{enumi})}
  \item $d(\bar{I}_V) \subset \bar{I}_V$ (i.e. $\bar{I}_V$ is a \textit{differential} ideal of $(\wedge V\tpow 2\otimes \wedge \overline{V}, d)$),
  \item $d\bar{x}_i \in \bar{I}_{V,i}$ for $1\leq i \leq p$, and
  \item the map $m$ induces a complex homomorphism $m'\colon (\bar{I}_V, d) \rightarrow (I_V, d)$ and a dga homomorphism $m''\colon ((\wedge V\tpow 2 \otimes \wedge \overline{V})/\bar{I}_V, d) \rightarrow (\wedge V / I_V, d)$, and these are quasi-isomorphisms
 \end{enumerate}
\end{defn}

Here, the ideal $\bar{I}_{V,i}$ is not necessarily a {\it differential} ideal.
To prove the existence of a nice relative Sullivan model, we need the following proposition.

\begin{prop}
 \label{prop_NiceRelSullModel}
 Let $\fdga{V}$ be a semi-pure Sullivan algebra and
 \[
  m\colon (\wedge V\tpow 2\otimes \wedge \overline{V}, d) \xrightarrow{\qis} \fdga{V}
 \]
 a nice relative Sullivan model for $\mu_{\wedge V}$.
 If $(\wedge V\otimes \wedge x, d)$ is a relative Sullivan algebra with $\deg{x}\geq 2$ which is semi-pure as an absolute Sullivan algebra,
 then the followings hold.
 \begin{enumerate}
  \item If $\deg{x}$ is even, some $x' \in \bar{I}_V$ satisfies $dx'=\snd{dx} - \fst{dx}$ and $mx'=0$.
  \item If $\deg{x}$ is even, choose $x'$ as in (1) of this proposition.
	If $\deg{x}$ is odd, choose $x'$ as in (1) of \protect \MakeUppercase {P}roposition\nobreakspace \ref {prop_ConstructionOfRelSullModel}.
	Using this $x'$, define
	\[
	 n\colon ((\wedge V \otimes \wedge x)\tpow{2} \otimes (\wedge \overline{V} \otimes \wedge \bar{x}), d) \xrightarrow{\qis} (\wedge V \otimes \wedge x, d)
	\]
	by \protect \MakeUppercase {P}roposition\nobreakspace \ref {prop_ConstructionOfRelSullModel} (2).
	Then $n$ is also nice.
 \end{enumerate}
\end{prop}
\begin{proof}
 (1)
 Since $(\wedge V \otimes \wedge x, d)$ is semi-pure, $\snd{dx} - \fst{dx} \in I_V \otimes I_V \subset \bar{I}_V$.
 This element satisfies $m'(\snd{dx} - \fst{dx})=0$ and $d(\snd{dx} - \fst{dx})=0$.
 Hence, applying \protect \MakeUppercase {L}emma\nobreakspace \ref {lem_CoboundByKernel} for the surjective quasi-isomorphism $m'\colon (\bar{I}_V, d) \rightarrow (I_V, d)$, we obtain an element $x' \in \bar{I}_V$ with the above properties.
 
 (2)
 For simplicity of notation, we write $W=V\oplus \K x$.
 
 First, we consider the case $\deg{x}$ is odd.
 Since there are no additional elements of even degree, we have the conditions (a) and (b) in \protect \MakeUppercase {D}efinition\nobreakspace \ref {defn_NiceRelSullModel},
 and the map $n$ induces
 \[
  n'\colon (\bar{I}_{V\oplus \K x}, d) \rightarrow (I_{V\oplus \K x}, d)
 \]
 and
 \[
  n''\colon (((\wedge V \otimes \wedge x)\tpow{2} \otimes (\wedge \overline{V} \otimes \wedge \bar{x}))/\bar{I}_{V\oplus \K x}, d) \rightarrow ((\wedge V \otimes \wedge x)/I_{V\oplus \K x}, d).
 \]
 Using the isomorphisms
 \[
  ((\wedge V \otimes \wedge x)\tpow{2} \otimes (\wedge \overline{V} \otimes \wedge \bar{x}))/\bar{I}_{V\oplus \K x} \cong (\wedge V\tpow{2}\otimes\wedge\overline{V})/\bar{I}_V \otimes \wedge x \tpow{2} \otimes \wedge \bar{x}
 \]
 and
 \[
  (\wedge V \otimes \wedge x)/ I_{V\oplus\K x} \cong \wedge V / I_V \otimes \wedge x,
 \]
 the map $n''$ can be identified with a map
 \[
  n''\colon ((\wedge V\tpow{2}\otimes\wedge\overline{V})/\bar{I}_V \otimes \wedge x \tpow{2} \otimes \wedge \bar{x}, d) \rightarrow (\wedge V / I_V \otimes \wedge x, d).
 \]
 Under this identification, $n''$ is decomposed into $n'' = \beta\alpha$ as in the following diagram.
 Here, $\alpha$ and $\beta$ are defined as in \protect \MakeUppercase {P}roposition\nobreakspace \ref {prop_ConstructionOfRelSullModel}.
 Then, $\alpha$ and $\beta$ are quasi-isomorphisms by \protect \MakeUppercase {L}emma\nobreakspace \ref {lem_QisOfRelSullAlg} and by \protect \MakeUppercase {L}emma\nobreakspace \ref {lem_QisElementary}, respectively.
 Hence $n''$ is a quasi-isomorphism. 
 \[
 \xymatrix@C=0pt{
 ((\wedge V \otimes \wedge x)\tpow{2} \otimes (\wedge \overline{V} \otimes \wedge \bar{x}))/\bar{I}_{V\oplus \K x} \ar[rr]^-{n''} \ar[d]^{\cong}
 && (\wedge V \otimes \wedge x)/I_{V\oplus \K x} \ar[d]^{\cong}\\
 (\wedge V\tpow{2}\otimes\wedge\overline{V})/\bar{I}_V \otimes \wedge x \tpow{2} \otimes \wedge \bar{x} \ar[rr]^-{n''} \ar[dr]_\alpha
 && \wedge V / I_V \otimes \wedge x\\
 &\wedge V / I_V \otimes (\wedge x\tpow{2} \otimes \wedge\bar{x}) \ar[ur]_\beta
 }
 \]
 Then, applying the five lemma to the cohomology long exact sequence of the diagram
 \[
 \xymatrix{
 0 \ar[r] & \bar{I}_W \ar[r] \ar[d]^{n'} & \wedge W\tpow{2} \otimes \wedge \overline{W} \ar[r] \ar[d]_\qis^n & (\wedge W\tpow 2 \otimes \wedge \overline{W}) / \bar{I}_W \ar[r] \ar[d]_\qis^{n''} & 0\\
 0 \ar[r] & I_W \ar[r] & \wedge W \ar[r] & \wedge W / I_W \ar[r] & 0,
 }
 \]
 we prove $n'$ is a quasi-isomorphism.

 Next, we consider the case $\deg{x}$ is even.
 By the construction in \protect \MakeUppercase {P}roposition\nobreakspace \ref {prop_ConstructionOfRelSullModel},
 we have the properties (a) and (b) in \protect \MakeUppercase {D}efinition\nobreakspace \ref {defn_NiceRelSullModel} for $n$ and that $n$ induces $n'$ and $n''$.
 By the isomorphisms $((\wedge V\tpow{2} \otimes \wedge\overline{V}) / \bar{I}_V, d)  \cong  ((\wedge W\tpow{2} \otimes \wedge\bar{W}) / \bar{I}_W, d)$ and $(\wedge V / I_V, d) \cong (\wedge W / I_W, d)$, the map $n''$ is identified with $m''$, and hence $n''$ is a quasi-isomorphism.
 \[
 \xymatrix{
 (\wedge V\tpow{2} \otimes \wedge\overline{V}) / \bar{I}_V \ar[r]^\cong \ar[d]_\qis^{m''} & (\wedge W\tpow{2} \otimes \wedge\bar{W}) / \bar{I}_W \ar[d]^{n''}\\
 \wedge V / I_V \ar[r]^\cong & \wedge W / I_W
 }
 \]
 Hence $n'$ is also a quasi-isomorphism by the cohomology long exact sequence.
\end{proof}

\begin{cor}
 \label{cor_NiceRelSullModel}
 If $\fdga{V}$ is a semi-pure Sullivan algebra, then there is a nice relative Sullivan model for $\mu_{\wedge V}$.
\end{cor}
\begin{proof}
 Use \protect \MakeUppercase {P}roposition\nobreakspace \ref {prop_NiceRelSullModel} inductively on $\dim V$.
\end{proof}


\section{Construction of the shriek map}
\label{sect_ConstructionOfShriek}

In this section, let $\fdga{V}$ be a Sullivan algebra satisfying $\dim V < \infty$ and $V^1 = 0$.
And we fix a basis
$x_1, \ldots ,x_p, y_1, \ldots y_q$ of $V$
such that $\deg{x_i}$ is even and $\deg{y_j}$ is odd for each $i,j$,
and the corresponding basis
$\bar{x}_1, \ldots, \bar{x}_p, \bar{y}_1, \ldots \bar{y}_q$ of $\overline{V}$.
We will construct the shriek map $\shriek$ for $\fdga{V}$ by induction on $\dim V$.

\begin{defn}
 \label{defn_ConstructionOfShriek}
 Let $(\wedge V \otimes \wedge x, d)$ be a relative Sullivan algebra with $\deg{x} \geq 2$.
 Take $((\wedge V \otimes \wedge x)\tpow{2} \otimes (\wedge U \otimes \wedge \bar{x}), d)$ as in \protect \MakeUppercase {P}roposition\nobreakspace \ref {prop_ConstructionOfRelSullModel}.
 Define a $\K$-linear map
 \[
  \Phi\colon \hom_{\wedge V\tpow{2}} (\wedge V\tpow{2}\otimes \wedge U, \wedge V\tpow 2) \rightarrow \hom_{(\wedge V\otimes \wedge x)\tpow 2}((\wedge V\otimes \wedge x)\tpow 2 \otimes (\wedge U \otimes \wedge \bar{x}), (\wedge V\otimes \wedge x)\tpow 2)
 \]
 of degree $\deg{x}$, if $\deg{x}$ is odd, or of degree $1-\deg{x}$, if $\deg{x}$ is even, as follows:
 \begin{enumerate}
  \item In the case $\deg{x}$ is odd,
	for
	\[
	 f \in \hom_{\wedge V\tpow{2}} (\wedge V\tpow{2}\otimes \wedge U, \wedge V\tpow 2),
	\]
	define
	\[
	 \Phi(f) \in \hom_{(\wedge V\otimes \wedge x)\tpow 2}((\wedge V\otimes \wedge x)\tpow 2 \otimes (\wedge U \otimes \wedge \bar{x}), (\wedge V\otimes \wedge x)\tpow 2)
	\]
	by
	$\Phi(f)(u) = (\snd{x} - \fst{x})f(u) - (-1)^{\deg{f}}f(x'u)$ and
	$\Phi(f)(u\bar{x}^k)=0$ if $k \geq 1$,
	for $u \in \wedge U$.
  \item In the case $\deg{x}$ is even,
	for
	\[
	 f \in \hom_{\wedge V\tpow{2}} (\wedge V\tpow{2}\otimes \wedge U, \wedge V\tpow 2),
	\]
	define $\Phi(f)$ by
	$\Phi(f)(u\bar{x}) = (-1)^{\deg{f} + \deg{u}}f(u)$ and
	$\Phi(f)(u) = 0$,
	for $u \in \wedge U$.
 \end{enumerate}
\end{defn}

\begin{lem}
 \label{lem_ChainMap}
 The linear map $\Phi$ in \protect \MakeUppercase {D}efinition\nobreakspace \ref {defn_ConstructionOfShriek} is a chain map of odd degree.
 In other words, the linear map $\Phi$ satisfies $d\Phi = -\Phi d$.
\end{lem}
\begin{proof}
 This is proved by a straight-forward calculation from the definition of $\Phi$.
\end{proof}

To prove non-triviality of the cohomology class of a cocycle in $\hom_{\wedge V\tpow 2}(\wedge V\tpow 2 \otimes \wedge \overline{V}, \wedge V\tpow 2)$, we need the following notion.

\begin{defn}
 \label{defn_good}
 Let $(\wedge V\tpow 2\otimes \wedge \overline{V}, d)$ be the relative Sullivan algebra constructed in \protect \MakeUppercase {C}orollary\nobreakspace \ref {cor_ConstructionOfRelSullModel}.
 A cochain $f \in \hom_{\wedge V\tpow 2}(\wedge V\tpow 2 \otimes \wedge \overline{V}, \wedge V\tpow 2)$ is {\it good} if $f$ satisfies the following two conditions.
 \begin{enumerate}
  \renewcommand{\labelenumi}{(\alph{enumi})}
  \item There is an element $u$ of the ideal
	$\ideal{\fstandsnd{y_1}{y_1}, \cdots , \fstandsnd{y_q}{y_q}}{\wedge V\tpow 2}$
	such that
	\[
	f(\bar{x}_1\cdots\bar{x}_p) = \pm \prod_{j=1}^{q}(\snd{y_j} - \fst{y_j}) + u.
	\]
  \item The cochain $f$ vanishes on the ideal $\ideal{\bar{y}_1, \ldots , \bar{y}_q}{\wedge V\tpow 2 \otimes \wedge\overline{V}}$.
 \end{enumerate}
\end{defn}

The following proposition shows that the chain map $\Phi$ defined in \protect \MakeUppercase {D}efinition\nobreakspace \ref {defn_ConstructionOfShriek} can be used to construct a good cochain.

\begin{prop}
 \label{prop_GoodnessInduction}
 Let $(\wedge V\tpow 2\otimes \wedge \overline{V}, d)$ be the relative Sullivan algebra constructed in \protect \MakeUppercase {C}orollary\nobreakspace \ref {cor_ConstructionOfRelSullModel},
 and $(\wedge V \otimes \wedge x, d)$ a relative Sullivan algebra with $\deg{x} \geq 2$.
 If an element
 \[
  f \in \hom_{\wedge V\tpow 2}(\wedge V\tpow 2 \otimes \wedge \overline{V}, \wedge V\tpow 2)
 \]
 is good, then
 \[
  \Phi(f) \in \hom_{(\wedge V \otimes \wedge x)\tpow 2}((\wedge V \otimes \wedge x)\tpow 2 \otimes (\wedge \overline{V} \otimes \wedge \bar{x}), (\wedge V \otimes \wedge x)\tpow 2)
 \]
 is also good.
\end{prop}
\begin{proof}
 We can assume that there is an element $u \in \ideal{\fstandsnd{y_1}{y_1}, \cdots , \fstandsnd{y_q}{y_q}}{\wedge V\tpow 2}$ such that $\displaystyle f(\bar{x}_1\cdots\bar{x}_p) = \prod_{j=1}^{q} (\snd{y_j} - \fst{y_j}) + u$, since $f$ is good if and only if $-f$ is good.
 We denote by $J$ the ideal $\ideal{\fstandsnd{y_1}{y_1}, \cdots , \fstandsnd{y_q}{y_q}}{\wedge V\tpow 2}$.
 
 (1) Suppose $\deg{x}$ is odd.
 
 (a)
 By the construction of $\Phi$,
 \begin{eqnarray*}
  \Phi(f)(\bar{x}_1\cdots\bar{x}_p) &=& (\snd{x} - \fst{x})f(\bar{x}_1\cdots\bar{x}_p) - (-1)^{\deg{f}}f(x'\bar{x}_1\cdots\bar{x}_p)\\
  &=& (\snd{x} - \fst{x})\prod_{j=1}^{q} (\snd{y_j} - \fst{y_j}) \\
  &&+ (\snd{x} - \fst{x})u - (-1)^{\deg{f}}f(x'\bar{x}_1\cdots\bar{x}_p).
 \end{eqnarray*}
 Since $(\snd{x} - \fst{x})u \in J$, it is enough to show $f(x'\bar{x}_1\cdots\bar{x}_p) \in J$.
 Write $x' = \sum_k \alpha_k$ with each $\alpha_k$ is a monomial in $\fst{x_i}$, $\snd{x_i}$, $\bar{x}_i$, $\fst{y_j}$, $\snd{y_j}$, and $\bar{y}_j$.
 We will prove that $f(\alpha_k\bar{x}_1\cdots\bar{x}_p) \in J$ for all $k$.
 If $\alpha_k$ has $\bar{x}_i$ as a factor, then $f(\alpha_k\bar{x}_1\cdots\bar{x}_p) = 0$ by $\bar{x}_i^2 = 0$.
 If $\alpha_k$ has $\bar{y}_j$ as a factor, then $f(\alpha_k\bar{x}_1\cdots\bar{x}_p) = 0$ by $\alpha_k\bar{x}_1\cdots\bar{x}_p \in \ideal{\bar{y}_1, \ldots , \bar{y}_q}{\wedge V\tpow 2 \otimes \wedge\overline{V}}$.
 If $\alpha_k$ does not have $\bar{x}_i$ nor $\bar{y}_j$ as a factor, $\alpha_k$ has $\fst{y_l}$ or $\snd{y_l}$ as a factor, since $\deg{\alpha_k}$ is odd.
 If $\alpha_k$ has $\fst{y_l}$ as a factor, we can write $\alpha_k = \beta\fst{y_l}$ with $\beta \in \wedge V \tpow{2}$.
 Then,
 \begin{eqnarray*}
  f(\alpha_k\bar{x}_1\cdots\bar{x}_p)
   &=& \pm \beta\fst{y_l} f(\bar{x}_1\cdots\bar{x}_p)\\
  &=& \pm \beta\fst{y_l} \prod_j(\snd{y_j} - \fst{y_j}) \pm \beta\fst{y_l}u\\
  &=& \pm \beta\fstandsnd{y_l}{y_l} \prod_{j\neq l}(\snd{y_j} - \fst{y_j}) \pm \beta\fst{y_l}u\\
  &\in& J.
 \end{eqnarray*}
 If $\alpha_k$ has $\snd{y_l}$ as a factor, $f(\alpha_k\bar{x}_1\cdots\bar{x}_p) \in J$ is proved similarly.
 This proves $f(x'\bar{x}_1\cdots\bar{x}_p) \in J$.
 
 (b)
 The ideal $\ideal{\bar{y}_1, \ldots , \bar{y}_q, \bar{x}}{(\wedge V \otimes \wedge x)\tpow 2 \otimes (\wedge\overline{V} \otimes \wedge \bar{x})}$ is generated over $(\wedge V\otimes \wedge x)\tpow{2}$ by the elements $\alpha \bar{y}_j$ and $\beta \bar{x}^k$ with $\alpha, \beta \in \wedge \overline{V}$ and $k\geq 1$.
 It is easily proved from the definition that $\Phi(f)(\alpha\bar{y}_j)=0$ and $\Phi(f)(\beta\bar{x}^k)=0$.
 Hence $\Phi(f)$ vanishes on the ideal $\ideal{\bar{y}_1, \ldots , \bar{y}_q, \bar{x}}{(\wedge V \otimes \wedge x)\tpow 2 \otimes (\wedge\overline{V} \otimes \wedge \bar{x})}$.

 (2) Suppose $\deg{x}$ is even.
 Then the goodness of $\Phi(f)$ is trivial by the construction of $\Phi$.
\end{proof}

Then we obtain the existence of a good cocycle.

\begin{cor}
 \label{cor_Goodness}
 Take $(\wedge V\tpow 2\otimes \wedge \overline{V}, d)$ as in \protect \MakeUppercase {C}orollary\nobreakspace \ref {cor_ConstructionOfRelSullModel}.
 Then there is a good cocycle $f \in \hom_{\wedge V\tpow 2}(\wedge V\tpow 2 \otimes \wedge\overline{V}, \wedge V\tpow 2)$.
\end{cor}
\begin{proof}
 Use \protect \MakeUppercase {P}roposition\nobreakspace \ref {prop_GoodnessInduction} inductively on $\dim V$.
\end{proof}

Now, we prove non-triviality of the cohomology class of a good cocycle.

\begin{thm}
 \label{thm_NonTrivialityOfShriek}
 Let $\fdga{V}$ be a semi-pure Sullivan algebra and
 \[
  m\colon (\wedge V\tpow 2\otimes \wedge \overline{V}, d) \xrightarrow{\qis} \fdga{V}
 \]
 a nice relative Sullivan model for $\mu_{\wedge V}$.
 Then
 \begin{enumerate}
  \item There is a cocycle $f \in \hom_{\wedge V\tpow 2}(\wedge V\tpow 2 \otimes \wedge \overline{V}, \wedge V\tpow 2)$ satisfying the condition (a) of \protect \MakeUppercase {D}efinition\nobreakspace \ref {defn_good}.
  \item Let $f$ be a cocycle as in (1).
	Then the cohomology class
	\[
	 [f] \in \cohom{\hom_{\wedge V\tpow 2}(\wedge V\tpow 2\otimes \wedge \overline{V}, \wedge V\tpow 2)} = \ext_{\wedge V\tpow 2}(\wedge V, \wedge V\tpow 2)
	\]
	is non-trivial.
	In particular, this cocycle represents the shriek map $\shriek \in \ext_{\wedge V\tpow 2}(\wedge V, \wedge V\tpow 2)$.
 \end{enumerate}
\end{thm}
\begin{proof}
 (1)
 This is an immediate consequence of \protect \MakeUppercase {C}orollary\nobreakspace \ref {cor_Goodness}.

 (2)
 Since $\fdga{V}$ is semi-pure, the quotient dga $(\wedge V / I_V, d)$ can be defined.
 Recall the dga homomorphism $\varepsilon\cdot\id\colon \fdga{V}\tpow{2} \rightarrow \fdga{V}$, defined by $(\varepsilon\cdot\id)(1\otimes v)=v$ and $(\varepsilon\cdot\id)(v\otimes 1)=0$ for $v \in V$.
 We denote by $\pr\colon \fdga{V} \rightarrow (\wedge V / I_V, d)$ the quotient homomorphism.
 Consider the evaluation map
 \[
  \ev\colon \ext_{\wedge V\tpow{2}}(\wedge V, \wedge V\tpow{2}) \otimes \tor_{\wedge V\tpow{2}}(\wedge V, \wedge V / I_V) \rightarrow \tor_{\wedge V\tpow{2}}(\wedge V\tpow{2}, \wedge V / I_V),
 \]
 where
 $(\wedge V, d)\tpow{2}$, $(\wedge V, d)$, and $(\wedge V / I_V, d)$
 are $(\wedge V, d)\tpow{2}$-module via
 $\id$, $\varepsilon\cdot\id$, and $\pr$, respectively.
 To calculate $\ext$ and $\tor$, use the relative Sullivan model $m$ as the semi-free resolution of $\fdga{V}$.
 Then $\bar{x}_1\cdots\bar{x}_p\otimes 1$ is a cocycle in
 $(\wedge V\tpow{2} \otimes \wedge \overline{V}, d) \otimes_{\wedge V\tpow{2}} (\wedge V / I_V, d)$
 by the property (b) of a nice relative Sullivan model (see \protect \MakeUppercase {D}efinition\nobreakspace \ref {defn_NiceRelSullModel}).
 Take $u \in \ideal{\fstandsnd{y_1}{y_1}, \cdots , \fstandsnd{y_q}{y_q}}{\wedge V\tpow 2}$ such that
 \[
  f(\bar{x}_1\cdots\bar{x}_p) = \prod_{j=1}^{q}(\snd{y_j} - \fst{y_j}) + u,
 \]
 replacing $f$ by $-f$ if necessary.
 Consider the element
 \[
  \ev([f] \otimes [\bar{x}_1\cdots\bar{x}_p\otimes 1]) = [f(\bar{x}_1\cdots\bar{x}_p)\otimes 1] \in \tor_{\wedge V\tpow{2}}(\wedge V\tpow{2}, \wedge V / I_V).
 \]
 Since $\fst{y_j}$ and $u$ are in the kernel of $\varepsilon\cdot\id$,
 this element is mapped to $[y_1\cdots y_q] \in \cohom{\wedge V / I_V}$ by the canonical isomorphism
 \[
 \tor_{\wedge V\tpow{2}}(\wedge V\tpow{2}, \wedge V / I_V) \cong \cohom{\wedge V / I_V}.
 \]
 By the assumption that $V^1 = 0$, we have $[y_1\cdots y_q] \neq 0$.
 Hence we have
 \[
  [f] \neq 0 \in \ext_{\wedge V\tpow{2}}(\wedge V, \wedge V\tpow{2}).
 \]
 This proves the theorem.
\end{proof}

\begin{rmk}
 By a similar and somehow easier method, we can construct the non-trivial element in $\ext_{\wedge V}(\K, \wedge V)$, which appears in the definition of a Gorenstein algebra.
\end{rmk}

We need the following proposition in the proof of \protect \MakeUppercase {T}heorem\nobreakspace \ref {thm_main} in Section\nobreakspace \ref {sect_ProofOfMainThm}.

\begin{prop}
 \label{prop_ManyXBars}
 Let $(\fga[2]{V}\otimes\fga{\overline{V}}, d)$ be a relative Sullivan algebra constructed by \protect \MakeUppercase {C}orollary\nobreakspace \ref {cor_ConstructionOfRelSullModel},
 and $\repshriek \in \hom_{\fga[2]{V}}(\fga[2]{V}\otimes\fga{\overline{V}}, \fga[2]{V})$ a good cocycle constructed by \protect \MakeUppercase {C}orollary\nobreakspace \ref {cor_Goodness}.
 Then
 \[
  \mu_{\fga{V}} \circ \repshriek (\fga[2]{V} \otimes \wedge^n \overline{V^\mathrm{even}}) = 0
 \]
 for $n > p-q$.
\end{prop}
\begin{proof}
 It is enough to show
 $\mu_{\fga{V}} \circ \repshriek(\bar{x}_{i_1}\cdots\bar{x}_{i_n}) = 0$
 for $i_1,\ldots, i_n \in \{1,\ldots,p\}$ and $n > p-q$.
 We prove this by induction on $\dim V$.
 This is trivial when $\dim V = 0$.
 Assume $\dim V > 0$ and this is proved when $V$ has less dimension.
 By the assumptions, we can take a direct sum decomposition
 $V = W \oplus \K x$
 such that $\fdga{V} = (\fga{W}\otimes\fga{x}, d)$ is a relative Sullivan algebra over a Sullivan algebra $\fdga{W}$,
 $\repshriek' \in \hom_{\fga[2]{W}}(\fga[2]{W}\otimes\fga{\overline{W}}, \fga[2]{W})$
 with
 $\repshriek = \Phi(\repshriek')$,
 and a good relative Sullivan model $m\colon (\fga[2]{W}\otimes\fga{\overline{W}},d) \xrightarrow{\qis} \fdga{W}$ of $\mu_{\fga{W}}$.
 
 First, consider the case $\deg{x}$ is even.
 We may assume
 $x = x_p$,
 $W^\mathrm{even} = \K\{x_1,\ldots, x_{p-1}\}$, and
 $W^\mathrm{odd} = V^\mathrm{odd} = \K\{y_1,\ldots, y_q\}$.
 If $i_k \leq p-1$ for all $k$,
 \[
 \repshriek(\bar{x}_{i_1}\cdots\bar{x}_{i_n})
 = \Phi(\repshriek')(\bar{x}_{i_1}\cdots\bar{x}_{i_n})
 = 0
 \]
 by the definition of $\Phi$.
 If $i_k = p$ for some $k$, say $n$,
 \[
 \mu_{\fga{V}}\repshriek(\bar{x}_{i_1}\cdots\bar{x}_{i_{n-1}}\bar{x})
 = \pm \mu_{\fga{W}}\repshriek'(\bar{x}_{i_1}\cdots\bar{x}_{i_{n-1}})
 = 0
 \]
 since $n-1 > (p-1) - q$.
 
 Next, consider the case $\deg{x}$ is odd.
 We may assume $x=y_q$, $W^\mathrm{even}=V^\mathrm{odd}$, and $W^\mathrm{odd}=\K\{y_1,\ldots, y_q\}$.
 Let $n > p-q$.
 Take $x' \in (\fga[2]{W}\otimes\fga{\overline{W}}, d)$ as in \protect \MakeUppercase {P}roposition\nobreakspace \ref {prop_ConstructionOfRelSullModel} (1).
 Since $\mu_{\fga{V}}(\snd{x}-\fst{x})=0$,
 \begin{eqnarray*}
  \mu_{\fga{V}}\repshriek(\bar{x}_{i_1}\cdots\bar{x}_{i_n})
   &=& \mu_{\fga{V}}((\snd{x}-\fst{x})\repshriek'(\bar{x}_{i_1}\cdots\bar{x}_{i_n}))
   \pm \mu_{\fga{V}}\repshriek'(x'\bar{x}_{i_1}\cdots\bar{x}_{i_n})\\
  &=& \pm \mu_{\fga{V}}\repshriek'(x'\bar{x}_{i_1}\cdots\bar{x}_{i_n}).
 \end{eqnarray*}
 Since $\ker m = (\ker \mu_{\fga{W}})\oplus(\fga[2]{V}\otimes\wedge^+\overline{V})$,
 the element $x'\in\ker m$ can be written as
 \[
  x' = a + \sum_{l \geq 1}\sum_{j_1,\ldots,j_l}b_{j_1\cdots j_l}\bar{x}_{j_1}\cdots\bar{x}_{j_l} + \sum_k c_k\bar{y}_k,
 \]
 where $a \in \ker \mu_{\fga{W}}$, $b_{j_1\cdots j_l} \in \fga[2]{W}$, and $c_k \in \fga[2]{W}\otimes\fga{\overline{W}}$.
 Since $\repshriek'$ is $\fga[2]{W}$-linear,
 \[
 \begin{split}
  \repshriek'(x'\bar{x}_{i_1}\cdots\bar{x}_{i_n})
  = -a\repshriek'(\bar{x}_{i_1}\cdots\bar{x}_{i_n})
  &+ \sum_{l \geq 1}\sum_{j_1,\ldots,j_l}\pm b_{j_1\cdots j_l}\repshriek'(\bar{x}_{j_1}\cdots\bar{x}_{j_l}\bar{x}_{i_1}\cdots\bar{x}_{i_n})\\
  &+ \sum_k \repshriek'(c_k\bar{y}_k\bar{x}_{i_1}\cdots\bar{x}_{i_n}).
 \end{split}
 \]
 Then, we have
 \[
  \mu_{\fga{W}}(a\repshriek'(\bar{x}_{i_1}\cdots\bar{x}_{i_n})) = \mu_{\fga{W}}(\repshriek'(\bar{x}_{j_1}\cdots\bar{x}_{j_l}\bar{x}_{i_1}\cdots\bar{x}_{i_n})) = \repshriek'(c_k\bar{y}_k\bar{x}_{i_1}\cdots\bar{x}_{i_n}) = 0
 \]
 by $a \in \ker \mu_{\fga{W}}$,
 $l+n \geq 1+n > p-(q-1)$, and
 the goodness of $\repshriek'$, respectively.
 Hence
 \[
  \mu_{\fga{V}}(\repshriek'(x'\bar{x}_{i_1}\cdots\bar{x}_{i_n}))
  = \mu_{\fga{W}}(\repshriek'(x'\bar{x}_{i_1}\cdots\bar{x}_{i_n}))
  = 0.
 \]
 This completes the induction and proves the proposition.
\end{proof}

\section{Proof of \protect \MakeUppercase {T}heorem\nobreakspace \ref {thm_main}}
\label{sect_ProofOfMainThm}
In this section, we assume $\fdga{V}$ is a Sullivan algebra satisfying $\dim V < \infty$ and $V^1 = 0$.

Let $\mpath = \mpathv{V}$ be the dga $(\fga[2]{V}\otimes\fga{\overline{V}}, d)$ in \protect \MakeUppercase {C}orollary\nobreakspace \ref {cor_ConstructionOfRelSullModel},
and let $\mloop = \mloopv{V}$ denote $\fdga{V} \otimes_{\fga[2]{V}} \mpathv{V}$.
Consider the surjective quasi-isomorphism
\[
 \bar{\varepsilon}\otimes\id\colon \mpath \otimes_{\fga[2]{V}} \mloop[2]
 \xtwoheadrightarrow{\qis}
 \fdga{V} \otimes_{\fga[2]{V}} \mloop[2],
\]
which appears in the description of $\dlp$.
For $i=1,2$,
 we denote by $\iota_i$ the inclusion map from $\mloop$ into $\mpath \otimes_{\fga[2]{V}} \mloop[2]$,
 by $\iota'_i$ the inclusion map from $\mloop$ into $\fdga{V} \otimes_{\fga[2]{V}} \mloop[2]$ as the $i$-th factor of tensor products, 
 and by $\iota_0$ the inclusion map from $\mpath$ into $\mpath \otimes_{\fga[2]{V}} \mloop[2]$.

We need some lemmas to prove \protect \MakeUppercase {T}heorem\nobreakspace \ref {thm_main}.

\begin{lem}
 \label{lem_section}
 Assume that there is a direct sum decomposition $V = W \oplus \K x$ such that $\fdga{V} = (\fga{W} \otimes \fga{x}, d)$ is a relative Sullivan algebra with base $\fdga{W}$.
 Let
 \[
 \psi_0\colon \fdga{W} \otimes_{\fga[2]{W}} \mloopv[2]{W}
 \xrightarrow{\qis}
 \mpathv{W} \otimes_{\fga[2]{W}} \mloopv[2]{W}
 \]
 be a dga homomorphism making the diagram
 \[
 \xymatrix{
 \mloopv{W} \ar[d]^{\iota'_1} \ar[r]^-{\iota_1} & \mpathv{W} \otimes_{\fga[2]{W}} \mloopv[2]{W} \ar@{->>}[d]^{\bar{\varepsilon}\otimes\id}_\qis\\
 \fdga{W} \otimes_{\fga[2]{W}} \mloopv[2]{W} \ar[r]^-= \ar[ur]^-{\psi_0} & \fdga{W} \otimes_{\fga[2]{W}} \mloopv[2]{W}
 }
 \]
 commutative.
 Then there is a dga homomorphism
 \[
 \psi\colon
 \fdga{V} \otimes_{\fga[2]{V}} \mloopv[2]{V}
 \xrightarrow{\qis}
 \mpathv{V} \otimes_{\fga[2]{V}} \mloopv[2]{V}
 \]
 such that
 \begin{enumerate}
  \item \label{item_restriction}
       the restriction of $\psi$ to $\fdga{W} \otimes_{\fga[2]{W}} \mloopv[2]{W}$ equals to $\psi_0$,
  \item \label{item_xandy}
       there is an element
       $y\in \mpathv{W} \otimes_{\fga[2]{W}} \mloopv[2]{W}$
       satisfying
       $\psi\iota'_2(\bar{x}) = \iota_2(\bar{x}) + y$
       and
       $(\bar{\varepsilon}\otimes\id)y = 0$, and
  \item \label{item_diagram}
       the diagram
       \[
       \xymatrix{
       \mloopv{V} \ar[d]^{\iota'_1} \ar[r]^-{\iota_1} & \mpathv{V} \otimes_{\fga[2]{V}} \mloopv[2]{V} \ar@{->>}[d]^{\bar{\varepsilon}\otimes\id}_\qis\\
       \fdga{V} \otimes_{\fga[2]{V}} \mloopv[2]{V} \ar[r]^-= \ar[ur]^-{\psi} & \fdga{V} \otimes_{\fga[2]{V}} \mloopv[2]{V}
       }
       \]
	commutes.
 \end{enumerate}
 In particular, the map $\psi$ satisfies
 \[
  \im\psi \subset (\fga[2]{V} \otimes \wedge \overline{W}) \otimes_{\fga[2]{V}} \mloopv[2]{V}.
 \]
\end{lem}
\begin{proof}
 By the conditions (\ref{item_restriction}) and (\ref{item_diagram}),
 $\psi$ is already defined on $\fga{W}\otimes_{\fga[2]{W}}\mloopv[2]{W}$ and $\iota'_1(\mloopv[2]{V})$.
 Hence it is enough to define the image of $\iota'_2(\bar{x})$ to construct $\psi$.
 For the element $\bar{x}\in \mloopv{V}$, we have $d\bar{x}\in\mloopv{W}$.
 Hence we have $d\iota'_2(\bar{x}) \in \fga{W} \otimes_{\fga[2]{W}} \mloopv[2]{W}$
 and $d\iota_2(\bar{x}) \in \mpathv{W} \otimes_{\fga[2]{W}} \mloopv[2]{W}$.
 Then the element
 \[
  \alpha = \psi_0 d \iota'_2(\bar{x}) - d\iota_2(\bar{x}) \in \mpathv{W} \otimes_{\fga[2]{W}} \mloopv[2]{W}
 \]
 satisfies
 $d\alpha = 0$
 and
 $(\bar\varepsilon\otimes\id)\alpha=0$.
 Hence, by \protect \MakeUppercase {L}emma\nobreakspace \ref {lem_CoboundByKernel}, there is an element
 $y \in \mpathv{W} \otimes_{\fga[2]{W}} \mloopv[2]{W}$
 satisfying $dy=\alpha$ and $(\bar\varepsilon\otimes\id)y = 0$.
 Define $\psi$ by $\psi(\iota'_2(\bar{x})) = \iota_2(\bar{x}) + y$.
 Then $\psi$ is a dga homomorphism and satisfies the conditions.
\end{proof}

\begin{lem}
 \label{lem_TrivialityOfProduct}
 If there are a dga homomorphism
 \[
 \psi\colon
 \fdga{V} \otimes_{\fga[2]{V}} \mloop[2]
 \rightarrow
 \mpath \otimes_{\fga[2]{V}} \mloop[2]
 \]
 and a representative $\repshriek\in\hom_{\fga[2]{V}}(\mpath, \fga[2]{V})$ of $\shriek$ satisfying
 \[
 (\bar{\varepsilon}\otimes\id)\circ\psi = \id
 \colon \fdga{V} \otimes_{\fga[2]{V}} \mloop[2]
 \rightarrow \fdga{V} \otimes_{\fga[2]{V}} \mloop[2]
 \]
 and
 \[
  (\repshriek\otimes\id)\circ\psi = 0
 \colon \fdga{V} \otimes_{\fga[2]{V}} \mloop[2]
 \rightarrow \fdga[2]{V} \otimes_{\fga[2]{V}} \mloop[2],
 \]
 then $\dlp_{\fga{V}}$ is trivial.
\end{lem}
\begin{proof}
 Since $\bar{\varepsilon}\otimes\id$ and $\id$ are quasi-isomorphisms, the homomorphism $\psi$ is also a quasi-isomorphism.
 Hence $\cohom{\repshriek\otimes\id}\circ\cohom{\psi}=0$ implies $\cohom{\repshriek\otimes\id}=0$.
 Since $\cohom{\repshriek\otimes\id}$ appears in the description of $\dlp_{\fga{V}}$, this proves the lemma.
\end{proof}

\begin{lem}
 \label{lem_TrivialityOfCoproduct}
 If there is a representative $\repshriek\in\hom_{\fga[2]{V}}(\mpath, \fga[2]{V})$ of $\shriek$ satisfying $\mu_{\fga{V}}\circ \repshriek = 0$,
 then $\dlcop_{\fga{V}}$ is trivial.
\end{lem}
\begin{proof}
 Let $\iota$ be the inclusion map from $\fdga{V}$ into $\fdga{V} \otimes_{\fga[2]{V}} \mpath$.
 Then the diagram
 \[
 \xymatrix@C=70pt{
 \mpath\otimes_{\fga[2]{V}}(\mpath\otimes_{\fga[2]{V}}\mpath) \ar[d]^{\repshriek\otimes\id} \ar@/^50pt/[rddd]^{(\iota\circ\mu_{\fga{V}}\circ \repshriek)\cdot(\bar{\varepsilon}\otimes\id)}&\\
 \fga[2]{V}\otimes_{\fga[2]{V}}(\mpath\otimes_{\fga[2]{V}}\mpath) \ar[d]_\cong \ar@/^10pt/[rdd]^{(\iota\circ\mu_{\fga{V}})\cdot(\bar{\varepsilon}\otimes\id)}&\\
 \mpath\otimes_{\fga[2]{V}}\mpath \ar[d]^{\bar{\varepsilon}\otimes\id}_{\qis} &\\
 \fga{V} \otimes_{\fga[2]{V}} \mpath \ar[r]^= & \fga{V} \otimes_{\fga[2]{V}} \mpath\\
 }
 \]
 commutes.
  Since the composition of vertical maps appears in the last part of the description of $\dlcop$,
 this proves the lemma.

\end{proof}

\begin{proof}[Proof of \protect \MakeUppercase {T}heorem\nobreakspace \ref {thm_main}]
 (\ref{item_TrivTamanoi})
 Using the filtration of $W$ as a relative Sullivan algebra,
 take a basis $w_1,\ldots,w_n$ of $W$ satisfying $dw_i \in \fga{x}\otimes\wedge(w_1,\ldots,w_{i-1})$.
 Let $V(i)$ be a subspace $\K\{x,w_1,\ldots,w_i\}$ of $V$ and $J_i$ an ideal $\ideal{\snd{x} - \fst{x}}{\fga[2]{V(i)}}$ of $\fga[2]{V(i)}$.
 Define a $(\wedge V(0), 0)\tpow{2}$-linear map
 \[
  \repshriek^{(0)}\colon \mpathv{V(0)} \rightarrow (\wedge V(0), 0)\tpow{2}
 \]
 by $\repshriek^{(0)}(1) = \snd{x}-\fst{x}$ and $\repshriek^{(0)}(x^k) = 0$ for $k \geq 1$,
 and $\fdga[2]{V(i)}$-linear maps
 \[
  \repshriek^{(i)}\colon \mpathv{V(i)} \rightarrow \fdga[2]{V(i)}
 \]
 by $\repshriek^{(i)}=\Phi(\repshriek^{(i-1)})$ inductively on $i$, for $1 \leq i \leq n$.
 Then, we have $\im \repshriek^{(i)} \subset J_i$ inductively, and hence $\mu_{\fga{V(i)}}\circ\repshriek^{(i)} = 0$.
 In particular, we have $\mu_{\fga{V}}\circ\repshriek^{(n)} = 0$.
 Hence \protect \MakeUppercase {L}emma\nobreakspace \ref {lem_TrivialityOfCoproduct} implies $\dlcop_{\fga{V}} = 0$.

 (\ref{item_TrivFT})
 Let $\repshriek' \in \hom_{\fga[2]{W}}(\mpathv{W}, \fga[2]{W})$ be a good cocycle and define $\repshriek = \Phi(\repshriek')$.
 Take $\psi$ as in \protect \MakeUppercase {L}emma\nobreakspace \ref {lem_section}.
 Since $\repshriek(\fga[2]{V} \otimes \wedge \overline{W}) = 0$ by the definition of $\Phi$,
 the composition $(\repshriek\otimes\id)\circ\psi$ vanishes.
 Hence \protect \MakeUppercase {L}emma\nobreakspace \ref {lem_TrivialityOfProduct} implies $\dlp_{\fga{V}}=0$.

 (\ref{item_TrivNaito})
 Since the difference $\dim W^{\mathrm{odd}} - \dim W^\mathrm{even}$ is homotopy invariant for Sullivan algebras $\fdga{W}$ with $\dim W < \infty$, we may assume $\fdga{V}$ is semi-pure by \protect \MakeUppercase {T}heorem\nobreakspace \ref {thm_Semi-pureModel}.
 Let $\repshriek \in \hom_{\fga[2]{V}}(\mpath, \fga[2]{V})$ be a good cocycle constructed by \protect \MakeUppercase {C}orollary\nobreakspace \ref {cor_Goodness}.
 The $\fga[2]{V}$-module $\mpath = \fga[2]{V}\otimes\fga{\overline{V}}$ is generated by the elements of the form
 $\bar{x}_{i_1}\cdots\bar{x}_{i_n}$ and $a\bar{y}_j$ for $a \in \fga{\overline{V}}$.
 For these elements, we have
 $\mu_{\fga{V}}\repshriek(\bar{x}_{i_1}\cdots\bar{x}_{i_n}) = 0$ by $p-q = \dim V^\mathrm{even} - \dim V^\mathrm{odd} < 0$
 , and
 $\mu_{\fga{V}}\repshriek(a\bar{y}_j) = 0$ by the goodness of $\repshriek$.
 Hence we have $\mu_{\fga{V}}\repshriek = 0$, and this proves $\dlcop_{\fga{V}}=0$ using \protect \MakeUppercase {L}emma\nobreakspace \ref {lem_TrivialityOfProduct}.
\end{proof}

\appendix

\section{Appendix: Explicit construction of models}
\label{sect_appendix}
In this section, let $\fdga{V}$ be a Sullivan algebra with $V^1 = 0$,
and $\{V(k)\}_{k\geq -1}$ a filtration of $V$ for the Sullivan algebra $\fdga{V}$ with $V(-1) = 0$.
We will construct a relative Sullivan model for the multiplication map $\mu_{\fga{V}}$ without assuming $\dim V < \infty$,
and construct the element $x'$ in \protect \MakeUppercase {P}roposition\nobreakspace \ref {prop_ConstructionOfRelSullModel} explicitly.
In the case $\fdga{V}$ is minimal, this construction is given in \cite[\S15(c) Example 1]{felix-halperin-thomas01}.
We give the construction without assuming the minimality and with more detailed proof.

Before starting the construction, we define the exponential of a derivation and give some lemmas without proofs.

\begin{defn}
 \label{defn_LocallyNilpotent}
 Let $A$ be a graded algebra.
 \begin{enumerate}
  \item A derivation $\theta\colon A \rightarrow A$ is called \textit{locally nilpotent}
 if, for any $a \in A$, there exists a positive integer $n$ such that $\theta^na=0$.
  \item For a locally nilpotent derivation $\theta\colon A \rightarrow A$ of degree 0, we define a linear map $e^\theta\colon A \rightarrow A$
 by $\displaystyle e^\theta a = \sum_{n\geq0}\frac{1}{n!}\theta^na$ for $a \in A$.
 \end{enumerate}
\end{defn}

Note that, since $\theta$ is locally nilpotent, the sum $\displaystyle\sum_{n\geq0}\frac{1}{n!}\theta^na$ is a finite sum for any $a \in A$.
This linear map satisfies the following basic properties.

\begin{lem}
 \label{lem_PropertiesOfExp}
 Let $A$ be a graded algebra and $\theta\colon A \rightarrow A$ a locally nilpotent derivation of degree $0$.
 \begin{enumerate}
  \item The linear map $e^\theta\colon A \rightarrow A$ is a graded algebra homomorphism.
  \item If $\rho\colon A \rightarrow A$ is a locally nilpotent derivation of degree 0 with $\theta\rho = \rho\theta$,
	then the derivation $\theta+\rho$ is also locally nilpotent and we have $e^{\theta+\rho} = e^\theta e^\rho = e^\rho e^\theta$.
	In particular, $e^\theta = (e^{-\theta})^{-1}$ is an isomorphism.
  \item If $d\colon A\rightarrow A$ is a derivation with $d\theta = \theta d$,
	then we have $de^\theta = e^\theta d$.
 \end{enumerate}
\end{lem}

The derivation $\theta$ will be constructed by the following lemma in this section.

\begin{lem}
 \label{lem_Bracket}
 Let $A$ be a graded algebra, and $d$ and $s$ derivations on $A$.
 Define $\theta = ds - (-1)^{\deg{d}\deg{s}}sd$.
 Then $\theta\colon A \rightarrow A$ is a derivation.
\end{lem}

Now, let us start the construction of a relative Sullivan model for $\mu_{\fga{V}}$.
We begin the construction by defining a dga inductively on $k$.

\newcommand{\property}[2]{$(#1)_{#2}$}
\newcommand{\properties}[1]{\property 1{#1}, \property 2{#1}, \property 3{#1}, and \property4{#1}}
\newcommand{\filt}[1]{\fga[2]{V(#1)}\otimes\fga{\overline{V(#1)}}}
\begin{prop}
 \label{prop_ConstructionOfDifferential}
 Define a derivation
 \[
  s_k\colon \fga[2]{V(k)} \otimes \fga{\overline{V(k)}} \rightarrow \fga[2]{V(k)} \otimes \fga{\overline{V(k)}}
 \]
 of degree $(-1)$ by $s_k(\fst{v})=s_k(\snd{v})=\bar{v}$ and $s_k(\bar{v})=0$ for $v \in V(k)$.
 By induction on $k$, we define a dga $(\filt{k}, d_k)$ with the following properties.
 \begin{itemize}
  \item[\property 1k] For any element $v \in V(k)$, there is a positive integer $n$ such that
	       \[
		(s_{k-1}d_{k-1})^{n-1}s_{k-1}d\fst{v} = 0.
	       \]
	       Hence an element
	       \[
		\sum_{n\geq 1} \frac{1}{n!}(s_{k-1}d_{k-1})^{n-1}s_{k-1}d\fst{v} \in \filt{k-1}
	       \]
	       is well-defined.
	       Hereafter, we denote this element simply by $\sum_{n\geq 1} \frac{1}{n!}(sd)^{n}\fst{v}$.
  \item[\property 2k] By \property 1k, we define an element
	       \[
		\snd{v} - \fst{v} - \sum_{n\geq 1} \frac{1}{n!}(sd)^{n}\fst{v} \in \fga[2]{V(k)}\otimes\fga{\overline{V(k-1)}}
	       \]
	       for $v\in V(k)$.
	       Then this element is a cocycle, where the differential is defined by
	       \[
		(\fga[2]{V(k)}\otimes\fga{\overline{V(k-1)}}, d) = \fdga[2]{V(k)} \otimes_{\fga{V(k-1)}} (\filt{k-1}, d_{k-1}).
	       \]
  \item[\property 3k] By \property 2k, we define a dga $(\filt{k}, d_k)$ by
	       \[
		d_k\bar{v} = \snd{v} - \fst{v} - \sum_{n\geq 1} \frac{1}{n!}(sd)^{n}\fst{v},
	       \]
	       extending the above dga $(\fga[2]{V(k)}\otimes\fga{\overline{V(k-1)}}, d)$.
	       Using this, we define a derivation $\theta_k = s_kd_k + d_ks_k$ of degree $0$ on $\filt{k}$.
	       Then, this satisfies $d_k\theta_k = \theta_kd_k$, $s_k\theta_k = \theta_ks_k$, and $\theta_k^n = (s_kd_k)^n + (d_ks_k)^n$ for $n \geq 1$.
	       Moreover, $\theta_k$ is locally nilpotent.
  \item[\property4k] By \property 3k, we define a dga isomorphism
	       \[
		e^{\theta_k} = \sum_{n\geq 0}\frac{1}{n!}{\theta_k}^n\colon (\filt{k}, d_k) \rightarrow (\filt{k}, d_k).
	       \]
	       Then, for any $b\in \fga{V(k)}$, we have $e^{\theta_k}(\fst{b}) = \snd{b}$.
	       
 \end{itemize}
\end{prop}
\begin{proof}
 We prove \properties{k} by induction on $k$.
 These are obvious when $k = -1$.
 Assume $k\geq 0$ and the properties \properties{k-1}.

 \property 1k
 Since ${s_{k-1}}^2=0$, we have
 \[
  (s_{k-1}d_{k-1})^{n-1}s_{k-1}d\fst{v} = {\theta_{k-1}}^{n-1}s_{k-1}d\fst{v} = 0
 \]
 for sufficiently large $n$ by \property 3{k-1}.

 \property 2k
 Since $d^2=0$, we have
 \begin{eqnarray*}
  d\left(\snd{v} - \fst{v} - \sum_{n\geq 1} \frac{1}{n!}(sd)^{n}\fst{v}\right)
   &=& \snd{(dv)} - \fst{(dv)} - \sum_{n\geq 1}\frac{1}{n!}(ds)^n\fst{(dv)} \\
  &=& \snd{(dv)} - \fst{(dv)} - \sum_{n\geq 1}\frac{1}{n!}{\theta_{k-1}}^n\fst{(dv)} \\
  &=& \snd{(dv)} - e^{\theta_{k-1}}\fst{(dv)} \\
  &=& 0
 \end{eqnarray*}
 by \property4{k-1}.

 \property 3k
 The first half of \property 3k is obvious.
 To prove the second half, it is sufficient to consider the case $a = \fst{v}, \snd{v}, \bar{v}$ for $v \in V(k)$, since $\theta_k$ is a derivation.
 Let $i$ be $1$ or $2$, and $v$ an element of $V(k)$.
 Then we have
 \begin{eqnarray*}
  {\theta_k}^m(\ith{v})
   &=& {\theta_{k-1}}^{m-1}s_{k-1}d\ith{v} + {\theta_k}^{m-1}\left(\snd{v}-\fst{v}-\sum_{n\geq1}\frac{1}{n!}(sd)^n\fst{v}\right) \\
  &=& {\theta_{k-1}}^{m-1}s_{k-1}d\ith{v}
   + {\theta_{k-1}}^{m-2}s_{k-1}(d\snd{v}-d\fst{v}) \\
   &&- {\theta_{k-1}}^{m-1}\sum_{n\geq1}\frac{1}{n!}(s_{k-1}d_{k-1})^{n-1}s_{k-1}d\fst{v}\\
  &=& {\theta_{k-1}}^{m-2}\Biggl( \theta_{k-1}s_{k-1}d\ith{v}
			   + s_{k-1}(d\snd{v}-d\fst{v})\\
			   &&\hspace{50pt}-\theta_{k-1}\sum_{n\geq1}\frac{1}{n!}(s_{k-1}d_{k-1})^{n-1}s_{k-1}d\fst{v}
			  \Biggr) \\
  &=& 0
 \end{eqnarray*}
 for sufficiently large $m$ by \property 3{k-1}.
 On the other hand, we have $\theta_k\bar{v}=0$ by ${s_k}^2=0$.
 These prove the property \property 3k.

 \property4k
 Since $e^{\theta_k}$ is an algebra homomorphism, it is sufficient to prove $e^{\theta_k}\fst{v}=\snd{v}$ for $v \in V(k)$.
 Since $(d_ks_k)^2\fst{v}=0$, we have
 \[
  \sum_{n\geq1}\frac{1}{n!}(d_ks_k)^n\fst{v} = d_k\bar{v} = \snd{v} - \fst{v} - \sum_{n\geq1}\frac{1}{n!}(s_kd_k)^n\fst{v}.
 \]
 Hence
 \begin{eqnarray*}
  e^{\theta_k}\fst{v}
   &=& \fst{v} + \sum_{n\geq1}\frac{1}{n!}(s_kd_k)^n\fst{v} + \sum_{n\geq1}\frac{1}{n!}(d_ks_k)^n\fst{v} \\
  &=& \fst{v} + \sum_{n\geq1}\frac{1}{n!}(s_kd_k)^n\fst{v} + \left(\snd{v} - \fst{v} - \sum_{n\geq1}\frac{1}{n!}(s_kd_k)^n\fst{v}\right) \\
  &=& \snd{v}.
 \end{eqnarray*}
 This completes the induction and proves the proposition.
\end{proof}

\begin{defn}
 \label{defn_PathModel}
 We define a dga $(\fga[2]{V}\otimes\fga{\overline{V}}, d)$ to be the union of $(\filt{k}, d_k)$ in \protect \MakeUppercase {P}roposition\nobreakspace \ref {prop_ConstructionOfDifferential}.
\end{defn}

We need the following lemmas to complete the construction of a relative Sullivan model of the multiplication map.

\begin{lem}
 \label{lem_ExpRelSullAlg}
 The dga $(\fga[2]{V}\otimes\fga{\overline{V}}, d)$ in \protect \MakeUppercase {D}efinition\nobreakspace \ref {defn_PathModel} is a relative Sullivan algebra over the Sullivan algebra $\fdga[2]{V}$.
\end{lem}
\begin{proof}
 Define a filtration $\{\overline{V}(k)\}$ by $\overline{V}(k) = \overline{V(k)}$, the subspace of $\overline{V}$ corresponding to $V(k)$.
 Then this filtration satisfies the definition of a relative Sullivan algebra.
\end{proof}

\begin{lem}
 \label{lem_QisElementaryManyBasis}
 Let $(B\otimes\wedge W, d)$ be a relative Sullivan algebra over a cdga $(B,d)$ with $W = W^{\geq 2}$ and $dW \subset B$.
 Define a relative Sullivan algebra $(B\otimes\wedge W\tpow{2}\otimes\wedge\overline{W}, d)$ over $(B,d)$ by
 $d(\snd{w}) = d(\fst{w}) = dw$ and $d\bar{w} = \snd{w} - \fst{w}$ for $w \in W$, where $dw\in B$ is considered as an element of $B\otimes\wedge W\tpow{2}\otimes\wedge\overline{W}$ by the canonical inclusion,
 and the dga homomorphism
 \[
 \beta\colon (B\otimes\wedge W\tpow{2}\otimes\wedge\overline{W}, d) \rightarrow (B\otimes\wedge W, d)
 \]
 by $\beta(b) = b$ for $b\in B$, $\beta(\snd{w}) = \beta(\fst{w}) = w$, and $\beta(\bar{w})=0$, for $w \in W$.
 Then $\beta$ is a quasi-isomorphism.
\end{lem}
\begin{proof}
 Construct the homotopy inverse to $\beta$ by a method similar to \protect \MakeUppercase {L}emma\nobreakspace \ref {lem_QisElementary}.
\end{proof}

Now, we complete the construction. 

\begin{thm}
 \label{thm_ExpRelSullModel}
 Let $(\fga[2]{V}\otimes\fga{\overline{V}}, d)$ be the relative Sullivan algebra in \protect \MakeUppercase {D}efinition\nobreakspace \ref {defn_PathModel}, and define a dga homomorphism
 \[
  m\colon (\fga[2]{V}\otimes\fga{\overline{V}}, d) \rightarrow \fdga{V}
 \]
 by $m(\fst{v})=m(\snd{v})=v$ and $m(\bar{v})=0$ for $v\in V$.
 Then $m$ is a relative Sullivan model for $\mu_{\fga{V}}$.
\end{thm}
\begin{proof}
 It is obvious that $m$ is a dga homomorphism whose restriction to $\fga[2]{V}$ is equal to $\mu_{\fga{V}}$.
 Hence it is enough to prove that $m$ is a quasi-isomorphism.
 Let $m_k$ be the restriction
 \[
  m_k\colon  (\filt{k},d) \rightarrow (\fga{V(k)}, d)
 \]
 of $m$.
 We prove that $m_k$ is a quasi-isomorphism by induction on $k$.
 This is obvious for $m_{-1}$.
 Assume that $m_{k-1}$ is a quasi-isomorphism.
 Take a subspace $V_k$ of $V(k)$ such that $V(k) = V(k-1) \oplus V_k$.
 Define a dga by
 \[
 \begin{split}
  (&\fga{V(k-1)}\otimes\fga[2]{V_k}\otimes\fga{\overline{V_k}}, d)\\
  &= (\fga{V(k-1)}, d) \otimes_{\filt{k-1}} (\filt{k},d)
 \end{split}
 \]
 and a dga homomorphism
 \[
  \beta\colon (\fga{V(k-1)}\otimes\fga[2]{V_k}\otimes\fga{\overline{V_k}}, d) \rightarrow \fga{V(k-1)}\otimes\fga{V_k}
 \]
 by $\beta(v)=v$, $\beta(\fst{w})=\beta(\snd{w})=w$ and $\beta(\bar{w})=0$ for $v\in V(k-1)$ and $w \in V_k$.
 Let $\alpha = m_{k-1}\otimes\id$.
 Then $\alpha$ and $\beta$ are quasi-isomorphism by \protect \MakeUppercase {L}emma\nobreakspace \ref {lem_QisOfRelSullAlg} and \protect \MakeUppercase {L}emma\nobreakspace \ref {lem_QisElementaryManyBasis}, respectively.
 Hence $m_k$ is a quasi-isomorphism by the following commutative diagram.
 \[
 \xymatrix@C=0pt{
 \filt{k} \ar[rr]^{m_k}\ar[d]^\cong && \fga{V(k)} \ar[d]^{\cong} \\
 (\filt{k-1})\otimes\fga[2]{V_k}\otimes\fga{\overline{V_k}} \ar[rd]_{\alpha} \ar[rr]^-{m_k} && \fga{V(k-1)}\otimes\fga{V_k}\\
 & \fga{V(k-1)}\otimes\fga[2]{V_k}\otimes\fga{\overline{V_k}} \ar[ru]_\beta &
 }
 \]
\end{proof}

Finally, we describe some Sullivan models of free loop spaces.

\begin{defn}
 Define a Sullivan algebra by
 \[
  (\fga{V}\otimes\fga{\overline{V}},\bar{d}) = \fdga{V} \otimes_{\fga[2]{V}}(\fga[2]{V}\otimes\fga{\overline{V}}, d)
 \]
 and denote the quotient map by
 \[
  \mu\otimes\id\colon (\fga[2]{V}\otimes\fga{\overline{V}}, d) \rightarrow (\fga{V}\otimes\fga{\overline{V}}, \bar{d}).
 \]
 Define a derivation
 \[
  \bar{s}\colon \fga{V}\otimes\fga{\overline{V}} \rightarrow \fga{V}\otimes\fga{\overline{V}}
 \]
 of degree $(-1)$ by $s(v) = \bar{v}$ and $s(\bar{v})=0$ for $v \in V$.
\end{defn}

\begin{rmk}
 If $\fdga{V}$ is a Sullivan model of a topological space $M$ with $\cohom{M}$ of finite type,
 then $(\fga{V}\otimes\fga{\overline{V}}, \bar{d})$ is a Sullivan model of the free loop space $LM$ by the Eilenberg-Moore theorem.
\end{rmk}

To describe the differential $\bar{d}$, we need the following lemma.

\begin{lem}
 \label{lem_PropertyOfDSbar}
 The derivations $\bar{d}$ and $\bar{s}$ satisfy
 $\bar{s}^2=0$,
 $(\mu\otimes\id) d = \bar{d}(\mu\otimes\id)$, and $(\mu\otimes\id) s = \bar{s}(\mu\otimes\id)$.
\end{lem}
\begin{proof}
 Obvious from the definitions.
\end{proof}

Now, we describe the differential $\bar{d}$.

\begin{prop}
 \label{prop_DescriptionOfModelOfLM}
 The derivations $\bar{d}$ and $\bar{s}$ satisfy $\bar{d}\bar{s} = -\bar{s}\bar{d}$.
 In particular, $\bar{d}\bar{v}$ is calculated by
 \[
  \bar{d}\bar{v} = \bar{d}\bar{s}v = -\bar{s}\bar{d}v = -\bar{s}(dv)
 \]
 from the differential $d$ in $\fdga{V}$.
\end{prop}
\begin{proof}
 Denote the restrictions of $\bar{d}$ and $\bar{s}$ by
 \[
  \bar{d}_k, \bar{s}_k\colon \fga{V(k)}\otimes\fga{\overline{V(k)}} \rightarrow \fga{V(k)}\otimes\fga{\overline{V(k)}}.
 \]
 We prove the following properties by induction on $k$:
 \begin{itemize}
  \item[$(1)_k$] $(\bar{s}_k\bar{d}_k)^2v = 0$ holds for any $v\in V(k)$, and
  \item[$(2)_k$] $\bar{d}_k\bar{s}_k + \bar{s}_k\bar{d}_k = 0$ holds.
 \end{itemize}
 These are obvious for $k = -1$.
 Assume $(1)_{k-1}$ and $(2)_{k-1}$.

 $(1)_k:$
 By $(2)_{k-1}$, we have
 \[
  (\bar{s}_k\bar{d}_k)^2v = \bar{s}_{k-1}\bar{d}_{k-1}\bar{s}_{k-1}dv = -\bar{s}_{k-1}\bar{s}_{k-1}\bar{d}_{k-1}dv = 0.
 \]

 $(2)_k:$
 Since $\bar{d}_k\bar{s}_k + \bar{s}_k\bar{d}_k$ is a derivation, it is sufficient to prove $(\bar{d}_k\bar{s}_k + \bar{s}_k\bar{d}_k)v = 0$ for $v \in V(k)$.
 By \protect \MakeUppercase {L}emma\nobreakspace \ref {lem_PropertyOfDSbar} and $(1)_k$, we have
\begin{eqnarray*}
 \bar{d}_k\bar{s}_k v
  &=& \bar{d}_k\bar{s}_k(\mu\otimes\id)\fst{v} = (\mu\otimes\id)d_ks_k\fst{v} \\
 &=& (\mu\otimes\id)\left( \snd{v} - \fst{v} - \sum_{n\geq1}\frac{1}{n!}(s_kd_k)^n\fst{v}\right) \\
 &=& -\sum_{n\geq1}\frac{1}{n!}(\bar{s}_k\bar{d}_k)^nv = -\bar{s}_k\bar{d}_kv.
\end{eqnarray*}
 This completes the induction and proves the proposition.
\end{proof}

\begin{rmk}
 Consider the case $\dim V < \infty$.
 Since the element $\displaystyle x' = \sum_{n\geq1}\frac{1}{n!}(sd)^n\fst{x}$ satisfies the condition (1) of \protect \MakeUppercase {P}roposition\nobreakspace \ref {prop_ConstructionOfRelSullModel},
 the relative Sullivan model constructed by \protect \MakeUppercase {C}orollary\nobreakspace \ref {cor_ConstructionOfRelSullModel}
 and that by \protect \MakeUppercase {T}heorem\nobreakspace \ref {thm_ExpRelSullModel} coincide with each other.
\end{rmk}

\begin{rmk}
 We can naturally extend the definition of semi-purity (\protect \MakeUppercase {D}efinition\nobreakspace \ref {defn_Semipure}) and niceness (\protect \MakeUppercase {D}efinition\nobreakspace \ref {defn_NiceRelSullModel}) to a Sullivan algebra $\fdga{V}$ with $\dim V = \infty$.
 Then, it is easy to prove that the relative Sullivan model constructed by \protect \MakeUppercase {T}heorem\nobreakspace \ref {thm_ExpRelSullModel} is nice.
 Hence, by the preceding remark,
 we can use this model as the Sullivan model used  in Section\nobreakspace \ref {sect_ConstructionOfShriek} to describe the shriek map $\shriek$.
\end{rmk}

\section{Appendix: Pure and semi-pure Sullivan algebras}
\label{sect_pure}

In this appendix, we prove \protect \MakeUppercase {T}heorem\nobreakspace \ref {thm_Semi-pureModel}.
On the other hand, we have \protect \MakeUppercase {P}roposition\nobreakspace \ref {prop_PureMinimal}, which shows that the similar statement to \protect \MakeUppercase {T}heorem\nobreakspace \ref {thm_Semi-pureModel} does not hold for a pure Sullivan algebra.

To prove \protect \MakeUppercase {T}heorem\nobreakspace \ref {thm_Semi-pureModel}, we recall the constructions $\Cbar_*$, $\Cbar^*$ and $\Lcobar$ for differential graded Lie algebras and cocommutative differential graded coalgebras.
See \cite[Part IV]{felix-halperin-thomas01} for details of these constructions.
We say a differential graded coalgebra $(C,d)$ is \textit{1-connected} if $C = \K \oplus C_{\geq2}$.
Similarly, a differential graded Lie algebra $(L,d)$ is \textit{connected} if $L = L_{\geq1}$.
Here, we use the homological grading for differential graded Lie algebras and cocommutative differential graded coalgebras and cohomological one for dga's.

Let $(L,d)$ be a connected differential graded Lie algebra.
Define the suspension $sL$ of $L$ by $(sL)_n = L_{n-1}$.
The \textit{Cartan-Eilenberg-Chevalley construction} on $(L,d)$ is the 1-connected cocommutative differential graded coalgebra $\Cbar_*(L,d)=\fdga{sL}$,
where the differential on $\Cbar_*(L,d)$ is defined by the differential on $L$ and the Lie bracket of $L$.
Note that the differential satisfies $d(\wedge^isL)\subset \wedge^isL\oplus\wedge^{i-1}sL$.
Next, define the dga $\Cbar^*(L,d)$ to be the dual $\hom_\K(\Cbar_*(L,d), \K)=((\wedge sL)^\sharp, d)$ of the above construction, where $(-)^\sharp$ denotes the dual.
If $L$ is of finite type, then we have $\Cbar^*(L,d) \cong (\wedge(sL)^\sharp, d)$ and $\Cbar^*(L,d)$ is a Sullivan algebra.
Note that the differential satisfies $d((sL)^\sharp)\subset \wedge^1(sL)^\sharp\oplus\wedge^2(sL)^\sharp$, by the corresponding property for $\Cbar_*(L,d)$.

Let $(C,d)$ be a 1-connected cocommutative differential graded coalgebra.
Define a graded $\K$-module $\overline{C}$ to be $C_{\geq2}$ and its desuspension $s^{-1}\overline{C}$ by $(s^{-1}\overline{C})_n = C_{n+1}$.
The \textit{Quillen construction} on $(C,d)$ is the connected differential graded Lie algebra $\Lcobar(C,d)=(\L_{s^{-1}\overline{C}}, d)$,
where $\L_{s^{-1}\overline{C}}$ is the free graded Lie algebra on $s^{-1}\overline{C}$,
and the differential on $\Lcobar(C,d)$ is defined by the differential on $C$ and the comultiplication of $C$.

We will use the following properties of these constructions:
\begin{enumerate}
 \renewcommand{\labelenumi}{(\alph{enumi})}
 \item There is a quasi-isomorphism $(C,d) \xrightarrow{\qis} \Cbar_*(\Lcobar(C,d))$ of cocommutative differential graded coalgebras.
 \item If $\varphi\colon (L,d)\xrightarrow{\qis}(L',d)$ is a quasi-isomorphism of connected differential graded Lie algebras,
       then the induced map $\Cbar_*{\varphi}\colon \Cbar_*(L,d)\xrightarrow{\qis}\Cbar_*(L',d)$ is a quasi-isomorphism of cocommutative differential graded coalgebras.
 \item Let $\fdga{V}$ be a minimal Sullivan algebra with $V$ finite type and $V=V^{\geq2}$.
       Define $(L,d) = \Lcobar(\fdga{V}^\sharp)$.
       Then there is a natural isomorphism $V \cong (sH(L,d))^\sharp$ of graded $\K$-modules.
\end{enumerate}

To prove \protect \MakeUppercase {T}heorem\nobreakspace \ref {thm_Semi-pureModel}, we need some propositions.
The first one treats the finiteness of generators, and is essentially the same as \cite[\S12 (a) Example6]{felix-halperin-thomas01}.
\begin{prop}
 \label{prop_FiniteLieModel}
 Let $(L,d)$ be a differential graded Lie algebra with $L=L_{\geq1}$.
 For a positive integer $n$, we assume that $H_i(L)=0$ for any $i > n$.
 Then there is a differential ideal $I\subset L$ satisfying the following properties:
 \begin{itemize}
  \item the projection map $(L,d) \xrightarrow{\qis} (L/I, d)$ is a quasi-isomorphism and
  \item $(L/I)_i=0$ holds for any $i>n$.
 \end{itemize}
\end{prop}

\begin{proof}
 Since the coefficient $\K$ is a field,
 we can take a direct sum decomposition $L_n = H_n(L) \oplus (\im d)_n \oplus M_n$ as graded $\K$-modules
 such that $\ker d = H_n(L) \oplus (\im d)_n$ and $d|_{M_n}\colon M_n \xrightarrow{\cong} (\im d)_{n-1}$ is an isomorphism.
 Define a differential ideal $I$ by $I_i=0$ for $i<n$, $I_n=(\im d)_n$ and $I_j = L_j$ for $j>n$.
 Then it is easy to check the above properties.
\end{proof}

Using the constructions $\Cbar^*$ and $\Lcobar$ with the above proposition, we have the following one.

\begin{prop}
 \label{prop_QuadraticModel}
 Let $\fdga{V}$ be a Sullivan algebra satisfying $\dim V < \infty$ and $V^1 = 0$.
 Then there is a Sullivan algebra $\fdga{W}$ satisfying
 \begin{itemize}
  \item $\dim W < \infty$ and $W^1 = 0$,
  \item $dW \subset \wedge^{\leq2} W$, and
  \item $\fdga{W} \simeq \fdga{V}$.
 \end{itemize}
\end{prop}

\begin{proof}
 Taking the minimal model, we may assume $\fdga{V}$ is minimal.
 Define $(L,d) = \Lcobar(\fdga{V}^\sharp)$.
 Then $H(L)$ is finite dimensional by the property (c).
 Take a differential ideal $I\subset L$ by \protect \MakeUppercase {P}roposition\nobreakspace \ref {prop_FiniteLieModel}.
 Then, by the properties (a) and (b), we have
 $\Cbar^*(L/I,d) \xrightarrow{\qis} \Cbar^*(L,d) = \Cbar^*\Lcobar(\fdga{V}^\sharp) \xrightarrow{\qis}\fdga{V}$.
 Then $\fdga{W} = \Cbar^*(L/I)$ satisfies the desired properties.
\end{proof}

The key of the construction of a semi-pure model is the following proposition.
\begin{prop}
 \label{prop_ConstructSemi-pureModel}
 Let $\fdga{V}$ be a Sullivan algebra satisfying  $dV \subset \wedge^{\leq2} V$.
 Then there is a submodule $W$ of $V$ and a differential $\bar{d}$ on $\wedge W$ satisfying
 \begin{itemize}
  \item $(\wedge W, \bar{d})$ is semi-pure and 
  \item $(\wedge W, \bar{d})$ and $\fdga{V}$ are homotopy equivalent.
 \end{itemize}
\end{prop}
\begin{proof}
 Let $d_0\colon V \rightarrow V$ be the linear part of $d\colon \wedge V \rightarrow \wedge V$.
 Take direct sum decompositions $V^\mathrm{even}=U\oplus W^\mathrm{even}$ and $V^\mathrm{odd}=d_0U\oplus W^\mathrm{odd}$ such that $W^\mathrm{even}=\ker(d_0|_{V^\mathrm{even}})$ and $d_0\colon U \xrightarrow{\cong}d_0U$.
 Denote $W=W^\mathrm{even}\oplus W^\mathrm{odd}$.
 Then an inclusion
 $W\oplus U\oplus dU \rightarrow \wedge V$
 induces an isomorphism
 $\wedge(W\oplus U\oplus dU) \xrightarrow{\cong} \wedge V$
 of graded algebras \cite[Lemma 14.7]{felix-halperin-thomas01}.
 Define a differential on $\wedge(W\oplus U\oplus dU)$ by this isomorphism.
 Taking the quotient by the subalgebra $(\wedge(U\oplus dU),d)$, define
 $(\wedge W, \bar{d}) = (\wedge(W\oplus U\oplus dU), d) \otimes_{\wedge(U\oplus dU)} \K$.
 Since $H(\wedge(U\oplus dU))=\K$, we have $\fdga{V}\simeq (\wedge W, \bar{d})$.
 Moreover, $(\wedge W, \bar{d})$ is semi-pure by the definition of $W$.
 This completes the proof.
\end{proof}

Note that, if we replace the decomposition by $V = U \oplus d_0 U \oplus W$ with $d_0 U\oplus W = \ker d_0$ and $d_0|_U\colon U \xrightarrow{\cong} d_0 U$, then the resulting $(\wedge W, \bar{d})$ is the minimal Sullivan model of $\fdga{V}$ \cite[Theorem 14.9]{felix-halperin-thomas01}.
This construction is the key of the above proposition.
Moreover, this construction shows that the difference $\dim V^\mathrm{odd}-\dim V^\mathrm{even}$ is homotopy invariant for Sullivan algebras $\fdga{V}$ with $\dim V < \infty$.

Now it is easy to prove \protect \MakeUppercase {T}heorem\nobreakspace \ref {thm_Semi-pureModel} by the above propositions.

\begin{proof}
 [Proof of \protect \MakeUppercase {T}heorem\nobreakspace \ref {thm_Semi-pureModel}]
 This follows immediate from \protect \MakeUppercase {P}roposition\nobreakspace \ref {prop_QuadraticModel} and \protect \MakeUppercase {P}roposition\nobreakspace \ref {prop_ConstructSemi-pureModel}.
\end{proof}

\begin{rmk}
 Note that \protect \MakeUppercase {T}heorem\nobreakspace \ref {thm_Semi-pureModel} remains true
 if we assume $V$ is of finite type instead of assuming $\dim V < \infty$ in the theorem and the definition of semi-purity (\protect \MakeUppercase {D}efinition\nobreakspace \ref {defn_Semipure}).
 The only modifications of the proof are to omit \protect \MakeUppercase {P}roposition\nobreakspace \ref {prop_FiniteLieModel} and to replace $\Cbar^*(L/I, d)$ with $\Cbar^*(L, d)$ in the proof of \protect \MakeUppercase {P}roposition\nobreakspace \ref {prop_QuadraticModel}.
\end{rmk}

On the other hand, we consider pure Sullivan algebras.
The following proposition shows that the similar statement to \protect \MakeUppercase {T}heorem\nobreakspace \ref {thm_Semi-pureModel} does not hold for a pure Sullivan algebra.

\begin{prop}
 \label{prop_PureMinimal}
 Let $\fdga{V}$ be a pure Sullivan algebra.
 Then there is a direct sum decomposition $V = W \oplus U$ and differentials $\bar{d}_W$ and $\bar{d}_U$ on $\wedge W$ and $\wedge U$, respectively, satisfying the following properties.
 \begin{itemize}
  \item The Sullivan algebra $(\wedge W, \bar{d}_W)$ is pure and minimal.
  \item The Sullivan algebra $(\wedge U, \bar{d}_U)$ is pure and satisfies $\cohom[+]{\wedge U, \bar{d}_U} = 0$.
  \item There is an isomorphism $\fdga{V} \cong (\wedge W, \bar{d}_W) \otimes (\wedge U, \bar{d}_U)$ of dga's.
 \end{itemize}
 In particular, there is a homotopy equivalence $\fdga{V} \simeq (\wedge W, \bar{d}_W)$ of Sullivan algebras.
\end{prop}

\begin{proof}
 Recall that we assume $\dim V < \infty$ in the definition of pure Sullivan algebra (\protect \MakeUppercase {D}efinition\nobreakspace \ref {defn_Pure}).
 Hence the proposition is proved by induction on $\dim V$ by the following lemma.
\end{proof}

\begin{lem}
 Let $\fdga{V} = (\wedge(x_0, \ldots, x_p, y_0, \ldots, y_q), d)$ be a pure Sullivan algebra satisfying the following properties.
 \begin{itemize}
  \item The degree $\deg{x_i}$ is even, and $\deg{y_j}$ is odd for any $i,j$.
  \item $dy_0 = x_0 - a$ for some $a \in \wedge(x_1, \ldots, x_p)$.
 \end{itemize}
 Define a pure Sullivan algebra
 $(\wedge W, \bar{d}) = (\wedge(x_1, \ldots, x_p, y_1, \ldots, y_q), \bar{d})$
 by $\bar{d}x_i = 0$ and $\bar{d}y_j = \eta(dy_j)$,
 where
 $\eta\colon \wedge(x_0, \ldots, x_p) \rightarrow \wedge(x_1, \ldots, x_p)$
 is an algebra homomorphism defined by $\eta(x_0) = a$ and $\eta(x_i) = x_i$ for $1\leq i \leq p$.
 Also define a pure Sullivan algebra $(\wedge(x_0, y_0), \bar{d})$ by $\bar{d}x_0 = 0$ and $\bar{d}y_0 = x_0$.

 Then there is an isomorphism $\fdga{V} \cong (\wedge W, \bar{d}) \otimes (\wedge(x_0, y_0), \bar{d})$ of dga's.
\end{lem}

\begin{proof}
 Considering $\eta(dy_j)$ as an element of $\wedge(x_0, \ldots, x_p)$ by the canonical inclusion of $\wedge(x_1, \ldots, x_p)$ into $\wedge(x_0, \ldots, x_p)$, a polynomial $dy_j - \eta(dy_j)$ in $x_0, \ldots, x_p$ has a root $x_0 = a$.
 Hence there is an element $b_j \in \wedge(x_0, \ldots, x_p)$ satisfying $dy_j - \eta(dy_j) = (x_0 - a)b_j$.
 Using this element, we define algebra homomorphisms
 \[
  \varphi\colon \fdga{V} \rightarrow (\wedge W, \bar{d}) \otimes (\wedge(x_0, y_0), \bar{d})
 \]
 and
 \[
  \psi\colon (\wedge W, \bar{d}) \otimes (\wedge(x_0, y_0), \bar{d}) \rightarrow \fdga{V}
 \]
 by
 $\varphi(x_0) = 1\otimes x_0 + a\otimes 1,\, \varphi(x_i) = x_i\otimes 1,\, \varphi(y_0) = 1\otimes y_0,\, \varphi(y_j) = y_j \otimes 1 + (1\otimes y_0)\varphi(b_j)$
 and
 $\psi(1\otimes x_0) = x_0 - a,\, \psi(x_i\otimes 1) = x_i,\, \psi(1\otimes y_0) = y_0,\, \psi(y_j\otimes 1) = y_j - y_0 b_j$
 for $1 \leq i \leq p$ and $1 \leq j \leq q$, respectively.
 Note that $\varphi$ is well-defined, since $b_j$ is an element of $\wedge(x_0, \ldots, x_p)$.
 Then, it is easy to prove that $\varphi$ and $\psi$ are dga homomorphisms satisfying $\psi \varphi = \id_{\wedge V}$ and $\varphi \psi = \id_{\wedge W \otimes \wedge(x_0, y_0)}$.
 This proves the lemma.
\end{proof}

\begin{rmk}
 Even if we drop the assumption $\dim V < \infty$ in the definition of a pure Sullivan algebra,
 \protect \MakeUppercase {P}roposition\nobreakspace \ref {prop_PureMinimal} remains true.
 This is proved by a method similar to that of \protect \MakeUppercase {P}roposition\nobreakspace \ref {prop_PureMinimal}, using the finiteness of the degrees of the elements appearing in each $dy_j$ instead of the finiteness of $\dim V$.
\end{rmk}

\section*{Acknowledgment}
I would like to express my gratitude to Katsuhiko Kuribayashi and Takahito Naito for productive discussions and valuable suggestions.
Furthermore, I would like to thank my supervisor Nariya Kawazumi for the enormous support and comments.
This work was supported by JSPS KAKENHI Grant Number 16J06349 and the Program for Leading Graduate School, MEXT, Japan.

\bibliographystyle{alphaabbrv} 

\end{document}